\documentclass[11pt]{article}

\usepackage{amsmath, amsthm}
\usepackage{amsmath, amsfonts}
\usepackage{amsmath, amssymb}
\usepackage{amsmath}
\usepackage[all]{xy}

\newtheorem{theorem}{Theorem}[subsection]
\newtheorem{definition}[theorem]{Definition}
\newtheorem{definition-lemma}[theorem]{Definition/Lemma}
\newtheorem{definition-explanation}[theorem]{Definition/Explanation}
\newtheorem{explanation-definition}[theorem]{Explanation/Definition}
\newtheorem{definition-fact}[theorem]{Definition/Fact}
\newtheorem{lemma}[theorem]{Lemma}
\newtheorem{lemma-definition}[theorem]{Lemma/Definition}
\newtheorem{proposition}[theorem]{Proposition}

\newtheorem{example}[theorem]{Example}
\newtheorem{example-definition}[theorem]{Example/Definition}

\newtheorem{definition-prototype}[theorem]{Definition-Prototype}

\numberwithin{equation}{subsection}

\newtheorem{stheorem}{Theorem}[section]

\newtheorem{sdefinition-lemma}[stheorem]{Definition/Lemma}
\newtheorem{sdefinition-explanation}[stheorem]{Definition/Explanation}
\newtheorem{sexplanation-definition}[stheorem]{Explanation/Definition}
\newtheorem{sdefinition-fact}[stheorem]{Definition/Fact}

\newtheorem{slemma-definition}[stheorem]{Lemma/Definition}
\newtheorem{sproposition}[stheorem]{Proposition}

\newtheorem{sexample}[stheorem]{Example}
\newtheorem{sexample-definition}[stheorem]{Example/Definition}

\newtheorem{sdefinition-prototype}[stheorem]{Definition-Prototype}
\newtheorem{squestion}[stheorem]{Question}

\newtheorem{ssdefinition-lemma}[sstheorem]{Definition/Lemma}
\newtheorem{ssdefinition-explanation}[sstheorem]{Definition/Explanation}
\newtheorem{ssexplanation-definition}[sstheorem]{Explanation/Definition}

\newtheorem{sslemma-definition}[sstheorem]{Lemma/Definition}

\newtheorem{ssexample-definition}[sstheorem]{Example/Definition}

\newtheorem{ssdefinition-prototype}[sstheorem]{Definition-Prototype}

\newcommand{\categoryAb}{\mbox{\it ${\cal A}$b}}
\newcommand{\Ext}{\mbox{\rm Ext}\,}
\newcommand{\Homsheaf}{\mbox{\it ${\cal H}$om}\,}
\newcommand{\Image}{\mbox{\it Im}\,}
\newcommand{\ModCategory}{\mbox{\it ${\cal M}$\!od}\,}
\newcommand{\boldSch}{\mbox{\bf Sch}\,}
\newcommand{\boldSm}{\mbox{\bf Sm}\,}
\newcommand{\Spec}{\mbox{\it Spec}\,}
\newcommand{\Sym}{\mbox{\it Sym}}

\newcommand{\dimm}{\mbox{\it dim}\,}
\newcommand{\gr}{\mbox{\it gr}\,}
\newcommand{\grade}{\mbox{\it grade}\,}
\newcommand{\llist}{\mbox{\it list}\,}
\newcommand{\pr}{\mbox{\it pr}}
\newcommand{\rank}{\mbox{\it rank}\,}

\begin{document}

\enlargethispage{24cm}

\begin{titlepage}

$ $

\vspace{-1.5cm} 

\noindent\hspace{-1cm}
\parbox{6cm}{\small February 2011}\
   \hspace{8cm}\
   \parbox[t]{5cm}{yymm.nnnn [math.AG]}

\vspace{2cm}

\centerline{\large\bf
 Algebraic cobordism of filtered vector bundles on varieties:}
\vspace{1ex}
\centerline{\large\bf
 Notes on a work of Lee and Pandharipande}

\bigskip

\vspace{3em}

\centerline{\large
  Chien-Hao Liu, \hspace{1ex}
  Yu-jong Tzeng, \hspace{1ex} and \hspace{1ex}
  Shing-Tung Yau
}

\vspace{6em}

\begin{quotation}
\centerline{\bf Abstract}

\vspace{0.3cm}

\baselineskip 12pt  
{\small
 The construction of
   double point cobordism groups of vector bundles on varieties
   in the work [Lee-P] (arXiv:1002.1500 [math.AG])
   of Yuan-Pin Lee and Rahul Pandharipande
  gives immediately double point cobordism groups
   of filtered vector bundles on varieties.
 We note also that
  among
  the four basic operations
   -- direct sum, tensor product, dual, and {\it ${\cal H}$om} --
   on vector bundles on varieties,
  only taking dual is compatible with double point cobordisms
   of vector bundles on varieties in general,
  by a demonstration
  on an example of vector bundles on Calabi-Yau $3$-folds.
 A question on
  refined and/or higher algebraic cobordisms of vector bundles on varieties
  is posed in the end.
} 
\end{quotation}

\vspace{20em}

\baselineskip 12pt
{\footnotesize
} 

\bigskip

\noindent {\small MSC number 2010:
 14F43, 14J60; 57R90, 14J32, 81T30.
} 

\bigskip

\baselineskip 10pt
{\scriptsize
\noindent{\bf Acknowledgements.}
We thank
 Si Li and Baosen Wu
  for discussions.
C.-H.L.\ thanks in addition
 Y.-j.\ and Jacob Lurie
 for related topic courses
 and Ling-Miao Chou for moral support.
} 

\end{titlepage}

\newpage
$ $

\vspace{-3em}

\centerline{\sc Algebraic Cobordism of Filtered Vector Bundles on Varieties}

\vspace{2em}


\begin{flushleft}
{\Large\bf 0. Introduction and outline.}
\end{flushleft}
In this note, we observe that
 the construction of double point cobordism groups of vector bundles
 on varieties in the work [Lee-P] (arXiv:1002.1500 [math.AG])
 of Yuan-Pin Lee and Rahul Pandharipande
 gives immediately double point cobordism groups
 of filtered vector bundles on varieties.
We note also that
 among
 the four basic operations
  -- direct sum, tensor product, dual, and {\it ${\cal H}$om} --
  on vector bundles on varieties,
 only taking dual is compatible with double point cobordisms
  of vector bundles on varieties in general,
 by a demonstration on an example of vector bundles on Calabi-Yau $3$-folds.
A question on
 refined and/or higher algebraic cobordisms of vector bundles on varieties
 is posed in the end.

\bigskip

\noindent
{\bf Convention.}
 Standard notations, terminology, operations, facts in
  (1) algebraic geometry; (2) algebraic cobordism of varieties;
  (3) algebraic cobordism of vector bundles on varieties
  can be found respectively in
  $\;$(1) [Fu: Appendix B]; (2) [L-M], [Lev-P]; (3) [Lee-P].
 \begin{itemize}
  \item[$\cdot$]
   All schemes are over ${\Bbb C}$,
    which can be replaced by a field of characteristic $0$ or
    any field that satisfies the resolution-of-singularities property
    of Hironaka;
   [Hi] and [L-M: Appendix].

  \item[$\cdot$]
   The convention for the notion of
    the {\it vector bundle} $E$ associated
     to a {\it locally free sheaf} ${\cal E}$
    and its {\it projectivization} ${\Bbb P}(E)$ follows Fulton [Fu]
    (instead of Grothendieck in Hartshorne's text book).
   We will write ${\cal E}$ and $E$ interchangeably in this note
    to save notations.

  \item[$\cdot$]
   For a product ${\Bbb P}^{n_1}\times\,\cdots\,\times{\Bbb P}^{n_s}$
    of projective spaces (over ${\Bbb C}$),
   we'll denote the line bundle
    $\pi_1^{\ast}({\cal O}_{{\Bbb P}^{n_1}}(d_1))\otimes\,
      \cdots\,\otimes\pi_s^{\ast}({\cal O}_{{\Bbb P}^{n_s}}(d_s))$
     thereon,
     where
      $\pi_i:{\Bbb P}^{n_1}\times\,\cdots\,\times{\Bbb P}^{n_s}
        \rightarrow {\Bbb P}^{n_i}$ is the projection map,
    by ${\cal O}_{{\Bbb P}^{n_1}\times\,\cdots\,\times{\Bbb P}^{n_s}}
         (d_1,\,\cdots\,,d_s)$
       or simply ${\cal O}_{}(d_1,\,\cdots\,,d_s)$.

  \item[$\cdot$]
   For a pair $(f:Y\rightarrow X,L)$,
    where $Y$ (resp.\ $X$; $f:Y\rightarrow X$; $L$)
     is a smooth quasi-projective variety
    (resp.\ a (separated) scheme (of finite type) over ${\Bbb C}$;
            a projective morphism;
           a line bundle on $Y$),
   $[f:Y\rightarrow X, L]$ represents,
    on one hand, a class in $\Omega_{\ast}(X)$ in [L-M] and,
    on the other hand, a class in $\omega_{\ast,1}(X)$ in [Lee-P].
   In Sec.~2 of the current note, we need such a presentation
    of a class in $\Omega_{\ast}(\bullet)$
    while studying $\omega_{\ast,1}(\bullet)$.
   {To} avoid confusion, we denote a class
     $[f:Y\rightarrow X, L_1,\,\cdots\,,L_s]\in \Omega_{\ast}(X)$ also by
     $[f:Y\rightarrow X; L_1,\,\cdots\,,L_s]$ whenever needed for clarity.

  \item[$\cdot$]
   This note is meant to be a footnote to [Lee-P];
   in particular, we assume readers are familiar
    with both the notations and the contents of ibidem.
 \end{itemize}

\bigskip

\bigskip

\begin{flushleft}
{\bf Outline.}
\end{flushleft}
{\small
\baselineskip 12pt  
\begin{itemize}
 \item[0.]
  Introduction.

 \item[1.]
  Algebraic cobordism of filtered vector bundles on varieties.
  \vspace{-.6ex}
  \begin{itemize}
   \item[1.1]
    Double point cobordism of lists of lists of line bundles
    on varieties.

   \item[1.2]
    Double point cobordism of filtered vector bundles on varieties.
  \end{itemize}

 \item[2.]
  Algebraic cobordism of and basic operations on vector bundles
   on varieties.
  \vspace{-.6ex}
  \begin{itemize}
   \item[$\cdot$]
    (Non)compatibility between basic operations and algebraic cobordisms.

   \item[$\cdot$]
    Question:
    Refined and higher algebraic cobordisms of vector bundles on varieties?
  \end{itemize}

\end{itemize}
} 

\newpage

\section{Algebraic cobordism of filtered vector bundles
         on varieties.}

We explain in this section that
 the construction of Lee and Pandharipande
 of algebraic cobordism of vector bundles on varieties
 in [Lee-P]
gives also
 algebraic codordism of filtered vector bundles on varieties.

\bigskip

\subsection{Double point cobordism of lists of lists of line bundles
            on varieties.}

We recall [Lee-P: Sec.~2] in this subsection in a format we need.
This is the foundation of the current note.

\begin{definition}
{\bf [double point degeneration over $0\in {\Bbb P}^1$].}
{\rm ([Lev-P: Sec.~0.3].)} {\rm
  Let $Y\in\boldSm_{\Bbb C}$ be (smooth, quasi-projective)
   of pure dimension.
  A morphism $\pi:Y\rightarrow{\Bbb P}^1$ is
   a {\it double point degeneration} over $0\in {\Bbb P}^1$
  if $\pi^{-1}(0)=A\cup B$
   with $A$ and $B$ smooth codimension-$1$ closed subscheme of $Y$
    that intersect transversely.
  $D:=A\cap B$ is the {\it double point locus} of $\pi$
   over $0\in{\Bbb P}^1$.
 (Here, $A$, $B$, and $D$ can be either empty
   or have several connected components.)
  Let $N_{D/A}$ and $N_{D/B}$ be the normal bundles of $D$
   in $A$ and $B$ respectively.
  Note that $N_{D/A}\otimes N_{D/B}\simeq O_D$ and
  the two projective bundles
   ${\Bbb P}(O_D\oplus N_{D/A})\rightarrow D$
   and ${\Bbb P}(O_D\oplus N_{D/B})\rightarrow D$ are isomorphic.
  Either one will be denoted ${\Bbb P}(\pi)\rightarrow D$.
}\end{definition}

\begin{definition} {\bf [partition list].} {\rm
 A {\it partition list}
  of {\it size} $n$ and {\it type} $r_1<r_2<\,\cdots\,<r_s$
  is a tuple
  $$
   (\lambda,\vec{m}^{\mbox{\boldmath $\cdot$}})\; :=\;
   (\lambda,
    (m_1,\,\cdots\,,\,m_{r_1}; m_{r_1+1},\,\cdots\,,\,m_{r_2};
     \,\cdots\,; m_{r_{s-1}+1},\,\cdots\,,\,m_{r_s}))
  $$
  where
   \begin{itemize}
    \item[$\cdot$]
     $\lambda$ is a partition of $n$,
     (whose length will be denoted by $l(\lambda)$),

    \item[$\cdot$]
     the ungrouped tuple
     $\vec{m}
      :=(m_1,\,\cdots\,,\,m_{r_1}, m_{r_1+1},\,\cdots\,,\,m_{r_2},
          \,\cdots\,, m_{r_{s-1}+1},\,\cdots\,,\,m_{r_s}))$
     associated to $\vec{m}^{\mbox{\boldmath $\cdot$}}$
     is a list with $m_i\ge 0$ whose union of {\it nonzero} parts
     is a sub-partition $\mu\subset\lambda$.
   \end{itemize}
 Denote by ${\cal P}_{n,1^{r_1<r_2<\,\cdots\,<r_s}}$ the set of
  all partition lists of size $n$ and type $r_1<r_2<\,\cdots\,<r_s$.
 Note that there is a canonical isomorphism of sets
  ${\cal P}_{n,1^{r_1<r_2<\,\cdots\,<r_s}}
    \stackrel{\sim}{\rightarrow} {\cal P}_{n,1^{r_s}}$
  by forgetting the grouping.
}\end{definition}

\begin{definition} {\bf [partition].} {\rm
 A {\it partition}
   of {\it size} $n$ and {\it type} $r_1<r_2<\,\cdots\,<r_s$
  is a tuple
  $$
   (\lambda,\vec{\mu}) :=\;
   (\lambda, (\mu_1,\,\mu_2,\,\cdots\,,\mu_s))
  $$
  where
   \begin{itemize}
    \item[$\cdot$]
     $\lambda$ is a partition of $n$,
     (whose length will be denoted by $l(\lambda)$),

    \item[$\cdot$]
     the union of $\mu_1,\,\cdots\,,\mu_s$ is a sub-partition
      $\mu\subset\lambda$
     with length $l(\mu_i)\le r_i-r_{i-1}\,$, for $i=1,\,\ldots\,,s\,$.
   \end{itemize}
 Denote by ${\cal P}_{n,r_1<r_2<\,\cdots\,<r_s}$ the set of
  all partitions of size $n$ and type $r_1<r_2<\,\cdots\,<r_s$.
}\end{definition}

\begin{example} {\bf [partition list and partition].}
{\rm
 The sets ${\cal P}_{4,1^{2<3}}$ and  ${\cal P}_{4,2<3}$
  are shown below:
 (Elements that are bold-faced are those contained in
  ${\cal P}^{\natural}_{\ast,1^{2<3}}$ and
  ${\cal P}^{\natural}_{\ast,2<3}$ respectively;
  cf.\ Theorem~1.1.10 and Theorem~1.2.4.)

 {\scriptsize
 $$
 \begin{array}{lcl}
  \mbox{\normalsize ${\cal P}_{4,1^{2<3}}$}  & =
   & \left\{
      \begin{array}{l}
       (4,\,(0,0;0))\,,\;
       \mbox{\boldmath $(4,\,(4,0;0))$}\,,\;
       \mbox{\boldmath $(4,\,(0,4;0))$}\,,\;
       \mbox{\boldmath $(4,\,(0,0;4))$}\,,         \\[.8ex]
       (31,\,(0,0;0))\,,\;                         \\[.8ex]
       (31,\,(3,0;0))\,,\;   (31,\,(0,3;0))\,,\;   (31,\,(0,0;3))\,, 
       (31,\,(1,0;0))\,,\;   (31,\,(0,1;0))\,,\;   (31,\,(0,0;1))\,,\;
                                                   \\[.8ex]
       \mbox{\boldmath $(31,\,(3,1;0))$}\,,\;
       \mbox{\boldmath $(31,\,(1,3;0))$}\,,\;
       \mbox{\boldmath $(31,\,(3,0;1))$}\,,\;
       \mbox{\boldmath $(31,\,(1,0;3))$}\,,\;
       \mbox{\boldmath $(31,\,(0,3;1))$}\,,\;
       \mbox{\boldmath $(31,\,(0,1;3))$}\,,     \\[.8ex]
       (22,\,(0,0;0))\,,                           \\[.8ex]
       (22,\,(2,0;0))\,,\;   (22,\,(0,2;0))\,,\;   (22,\,(0,0;2))\,,\;
       \mbox{\boldmath $(22,\,(2,2;0))$}\,,\;
       \mbox{\boldmath $(22,\,(2,0;2))$}\,,\;
       \mbox{\boldmath $(22,\,(0,2;2))$}\,,        \\[.8ex]
       (211,\,(0,0;0))\,,\;                        \\[.8ex]
       (211,\,(2,0;0))\,,\;  (211,\,(0,2;0))\,,\;  (211,\,(0,0;2))\,,\;
       (211,\,(1,0;0))\,,\;  (211,\,(0,1;0))\,,\;  (211,\,(0,0;1))\,,
                                                   \\[.8ex]
       (211,\,(2,1;0))\,,\;  (211,\,(1,2;0))\,,\;
       (211,\,(2,0;1))\,,\;  (211,\,(1,0;2))\,,\;
       (211,\,(0,2;1))\,,\;  (211,\,(0,1;2))\,,\;  \\[.8ex]
       (211,\,(1,1;0))\,,\,  (211,\,(1,0;1))\,,\;  (211,\,(0,1;1))\,,\;
       \mbox{\boldmath $(211,\,(2,1;1))$}\,,\;
       \mbox{\boldmath $(211,\,(1,2;1))$}\,,\;
       \mbox{\boldmath $(211,\,(1,1;2))$}\,,       \\[.8ex]
       (1111,\,(0,0;0))\,,\;
       (1111,\,(1,0;0))\,,\; (1111,\,(0,1;0))\,,\; (1111,\,(0,0;1))\,,
                                                   \\[.8ex]
       (1111,\,(1,1;0))\,,\; (1111,\,(1,0;1))\,,\; (1111,\,(0,1;1))\,,\;
       (1111,\,(1,1;1))
      \end{array}
     \right\}\;; \\[24ex]
  \mbox{\normalsize ${\cal P}_{4,2<3}$}  & =
   & \left\{
      \begin{array}{l}
       (4,\,(\emptyset,\emptyset))\,,\;
       \mbox{\boldmath $(4,\,(4,\emptyset))$}\,,\;
       \mbox{\boldmath $(4,\,(\emptyset,4))$}\,,\; \\[.8ex]
       (31,\,(\emptyset,\emptyset))\,,\;
       (31,\,(3,\emptyset))\,,\;  (31,\,(\emptyset,3))\,,\;
       (31,\,(1,\emptyset))\,,\;  (31,\,(\emptyset,1))\,,  \\[.8ex]
       \mbox{\boldmath $(31,\,(31,\emptyset))$}\,,\;
       \mbox{\boldmath $(31,\,(3,1))$}\,,\;
       \mbox{\boldmath $(31,\,(1,3))$}\,,                  \\[.8ex]
       (22,\,(\emptyset,\emptyset))\,,\;
       (22,\,(2,\emptyset))\,,\;  (22,\,(\emptyset,2))\,,\;
       \mbox{\boldmath $(22,\,(22,\emptyset))$}\,,\;
       \mbox{\boldmath $(22,\,(2,2))$}\,,                  \\[.8ex]
       (211,\,(\emptyset,\emptyset))\,,\;
       (211,\,(2,\emptyset))\,,\; (211,\,(\emptyset,2))\,,\;
       (211,\,(1,\emptyset))\,,\; (211,\,(\emptyset,1))\,, \\[.8ex]
       (211,\,(21,\,\emptyset))\,,\;
       (211,\,(2,1))\,,\;  (211,\,(1,2))\,,\;
       (211,\,(11,\,\emptyset))\,,\; (211,\,(1,1))\,,\;    \\[.8ex]
       \mbox{\boldmath $(211,\,(21,1))$}\,,\;
       \mbox{\boldmath $(211,\,(11,2))$}\,,                \\[.8ex]
       (1111,\,(\emptyset,\emptyset))\,,\; 
       (1111,\,(1,\emptyset))\,,\; (1111,\,(\emptyset,1))\,,\;
       (1111,\,(11,\emptyset))\,,\; (1111,\,(1,1))\,,\;
       (1111,\,(11,1))
      \end{array}
     \right\}\,.
 \end{array}
 $$
 }
}\end{example}

\bigskip

\begin{flushleft}
{\bf Double point cobordism of lists of lists of line bundles
     on varieties.}
\end{flushleft}

\begin{definition}
{\bf [list of lists of line bundles on smooth varieties over $X$].}
{\rm
 For $X\in\boldSch_{\Bbb C}$ and
  $0<r_1<r_2<\,\cdots\,<r_s$ a sequence of positive integers,
 let
  ${\cal M}_{n,1^{r_1<r_2<\,\cdots\,<r_s}}(X)$
  be the set of isomorphism classes of tuples
  $$
   (f:Y\rightarrow X,
    (L_1,\,\cdots\,,L_{r_1}; L_{r_1+1},\,\cdots\,,L_{r_2};
     \,\cdots\,; L_{r_{s-1}+1},\,\cdots\,,L_{r_s}))
  $$
  with
   $Y\in \boldSm_{\Bbb C}$ of dimension $n$,
   $f$ projective, and
   $$
    (L_1,\,\cdots\,,L_{r_1}; L_{r_1+1},\,\cdots\,,L_{r_2};
     \,\cdots\,; L_{r_{s-1}+1},\,\cdots\,,L_{r_s})
   $$
    an (ordered) list of line bundles on $Y$
    with a {\it grouping} specified by $0<r_1<r_2<\,\cdots\,<r_s$,
     as indicated by $\,;\,$'s
      that divide the list and
       turn it into an (ordered) {\it list of} (ordered) {\it lists}.
 Here,
  an {\it isomorphism}
   from
    $(f:Y\rightarrow X,
     (L_1,\,\cdots\,,L_{r_1}; L_{r_1+1},\,\cdots\,,L_{r_2};
      \,\cdots\,; L_{r_{s-1}+1},\,\cdots\,,L_{r_s}))$
   to
    $(f^{\prime}:Y^{\prime}\rightarrow X,
      (L^{\prime}_1,\,\cdots\,,L^{\prime}_{r_1};
       L^{\prime}_{r_1+1},\,\cdots\,,L^{\prime}_{r_2}; \,\cdots\,;
       L^{\prime}_{r_{s-1}+1},\,\cdots\,,L^{\prime}_{r_s}))$
  is a tuple
   $$
    (\phi,
     (\psi_1,\,\cdots\,,\psi_{r_1}; \psi_{r_1+1},\,\cdots\,,\psi_{r_2};
       \,\cdots\,; \psi_{r_{s-1}+1},\,\cdots\,,\psi_{r_s}))
   $$
  where
    $\phi:Y\rightarrow Y^{\prime}$
     is an isomorphism of $X$-schemes and
    $\psi_i: L_i \stackrel{\sim}{\rightarrow}
                                 \phi^{\ast}L^{\prime}_i$
     are isomorphisms of line bundles on $Y$, for $i=1,\,\ldots\,,r_s$.
 The {\it isomorphism class} of
  $$
   (f:Y\rightarrow X,
    (L_1,\,\cdots\,,L_{r_1}; L_{r_1+1},\,\cdots\,,L_{r_2};
     \,\cdots\,; L_{r_{s-1}+1},\,\cdots\,,L_{r_s}))
  $$
  in ${\cal M}_{n,1^{r_1<\,\cdots\,<r_s}}(X)$
  is denoted by
  $$
   [f:Y\rightarrow X,
    (L_1,\,\cdots\,,L_{r_1}; L_{r_1+1},\,\cdots\,,L_{r_2};
     \,\cdots\,; L_{r_{s-1}+1},\,\cdots\,,L_{r_s})]
  $$
  or $[f:Y\rightarrow X,(L_1,\,\cdots\,,L_{r_s})]$ in short,
   with the specified grouping assumed.

 The set ${\cal M}_{n,1^{r_1<\,\cdots\,<r_s}}(X)$
  is a {\it monoid} under the disjoint union of domains.
 Denote by ${\cal M}_{n,1^{r_1<\,\cdots\,<r_s}}(X)^+$
  the {\it group completion} of ${\cal M}_{n,1^{r_1<\,\cdots\,<r_s}}(X)$.
}\end{definition}

\begin{definition}
{\bf [double point relation over $X$].}
{\rm ([Lee-P: Sec.~2.1, Definition~4];
      also [Lev-P: Sec.~14.4] and [Tz1: Sec.~2.2].)} {\rm
 Let
  \begin{itemize}
   \item[$\cdot$]
    $Y\in\boldSm_{\Bbb C}$ be of pure dimension $n+1$,

   \item[$\cdot$]
    $g:Y\rightarrow X\times{\Bbb P}^1$
     be a projective morphism
      for which the composition
      $\pi:= \pr_2\circ g:Y\rightarrow {\Bbb P}^1$
      is a double point degeneration over $0\in {\Bbb P}^1$,

   \item[$\cdot$]
    $(L_1,\,\cdots\,,L_{r_1};\,\cdots\,; L_{r_{s-1}+1},\,\cdots\,,L_{r_s})$
     be a list of lists of line bundles
     with grouping specified by $0<r_1<\,\cdots\,<r_s$,

   \item[$\cdot$]
    $[A\rightarrow X,(L_{1,A},\,\cdots\,,L_{r_s,A})]$,
    $[B\rightarrow X,(L_{1,B},\,\cdots\,,L_{r_s,B})]$,
    $[{\Bbb P}(\pi)\rightarrow X,
      (L_{1,{\Bbb P}(\pi)},\,\cdots\,,L_{r_s,{\Bbb P}(\pi)})]$
     $\in {\cal M}_{n,1^{r_1<\,\cdots\,<r_s}}(X)^+$
     be obtained from the fiber $\pi^{-1}(0)$ and the morphism
     $\pr_1\circ g$.
  \end{itemize}
 Let $\zeta\in {\Bbb P}^1({\Bbb C})$ be a regular value of $\pi$.
 The associated {\it double point relation} over $X$
  is defined to be the element
  $$
   [Y_{\zeta},(L_{i,Y_{\zeta}})_i]\,
    -\, [A\rightarrow X,(L_{i,A})_i]\,
    -\, [B\rightarrow X,(L_{i,B})_i]\,
    +\, [{\Bbb P}(\pi)\rightarrow X,(L_{i,{\Bbb P}(\pi)})_i]
  $$
  in ${\cal M}_{n,1^{r_1<\,\cdots\,<r_s}}(X)^+$,
  where $Y_{\zeta}:= \pi^{-1}(\zeta)$.
}\end{definition}

\begin{definition}
{\bf [double point cobordism group
      $\omega_{n,1^{r_1<\,\cdots\,<r_s}}(X)$].}
{\rm ([Lee-P: Sec.~2.1].)}
{\rm
 Continuing the discussion, let
  $$
   {\cal R}_{n,1^{r_1<\,\cdots\,<r_s}}(X)\,
        \subset\, {\cal M}_{n,1^{r_1<\,\cdots\,<r_s}}(X)^+
  $$
  be the subgroup generated by all double point relations over $X$.
 The {\it double point cobordism group} of $X$
  for lists of lists of line bundles of grouping type
    $r_1<\,\cdots\,<r_s$
  is defined to be the quotient group
  $$
   \omega_{n,1^{r_1<\,\cdots\,<r_s}}(X)\;
   :=\; {\cal M}_{n,1^{r_1<\,\cdots\,<r_s}}(X)^+/
         {\cal R}_{n,1^{r_1<\,\cdots\,<r_s}}(X)\,.
  $$
 Note that the sum
  $$
   \omega_{\ast,1^{r_1<\,\cdots\,<r_s}}(X)\;
   :=\; \bigoplus_{n=0}^{\infty} \omega_{n,1^{r_1<\,\cdots\,<r_s}}(X)
  $$
  is an $\omega_{\ast}({\Bbb C})$-module via product.
}\end{definition}

The following lemma is immediate:

\begin{lemma} {\bf [no influence of decoration
               to group/$\omega_{\ast}({\Bbb C})$-module structure].}
 The natural forgetful map
  $$
   {\cal M}_{n,1^{r_1<\,\cdots\,<r_s}}(X)\;
   \longrightarrow\; {\cal M}_{n,1^{r_s}}(X)
  $$
  that sends
  $$
   (f:Y\rightarrow X,
    (L_1,\,\cdots\,,L_{r_1}; L_{r_1+1},\,\cdots\,,L_{r_2};
     \,\cdots\,; L_{r_{s-1}+1},\,\cdots\,,L_{r_s}))
  $$
  to
  $$
   (f:Y\rightarrow X,
    (L_1,\,\cdots\,,L_{r_1}, L_{r_1+1},\,\cdots\,,L_{r_2},
     \,\cdots\,, L_{r_{s-1}+1},\,\cdots\,,L_{r_s}))
  $$
  induces a canonical group isomorphism
  $$
   \omega_{n,1^{r_1<\,\cdots\,<r_s}}(X)\;
    \stackrel{\sim}{\longrightarrow}\;
     \omega_{n,1^{r_s}}(X)\,,
  $$
  which is also an $\omega_{\ast}({\Bbb C})$-module isomorphism.
\end{lemma}

\begin{theorem}
{\bf [$\omega_{n,1^{r_1<\,\cdots\,<r_s}}
                         ({\Bbb C})\otimes_{\Bbb Z}{\Bbb Q}$
      as ${\Bbb Q}$-vector space].}
{\rm ([Lee-P: Sec.~2.2, Theorem 9].)}
 {To} each
  $(\lambda,\vec{m}^{\mbox{\boldmath$\cdot$}})
  \in {\cal P}_{n,1^{r_1<r_2<\,\cdots\,<r_s}}$,
  associate an element
  $$
   \phi(\lambda,\vec{m}^{\mbox{\boldmath $\cdot$}})\;=\;
   [{\Bbb P}^{\lambda}\,,\,
     (L_{m_1},\,\cdots\,,L_{m_{r_1}}\,;\,
      L_{m_{r_1+1}},\,\cdots\,,L_{m_{r_2}}\,;\,\cdots\,;\,
      L_{m_{r_{s-1}+1}},\,\cdots\,,L_{m_{r_s}})]
  $$
  in $\omega_{n,1^{r_1<r_2<\,\cdots\,<r_s}}({\Bbb C})$,
 where
  \begin{itemize}
   \item[$\cdot$]
    ${\Bbb P}^{\,\lambda}
     := {\Bbb P}^{\,\lambda_1}\times\,
        \cdots\,\times{\Bbb P}^{\,\lambda_{l(\lambda)}}\,$
    is a product of projective spaces specified by $\lambda$,

   \item[$\cdot$]
    $(L_{m_1},\,\cdots\,,L_{m_{r_1}}\,;\,
      L_{m_{r_1+1}},\,\cdots\,,L_{m_{r_2}}\,;\,\cdots\,;\,
      L_{m_{r_{s-1}+1}},\,\cdots\,,L_{m_{r_s}}))$
    is a list of lists of line bundles on ${\Bbb P}^{\,\lambda}$
    defined as follows:
    \begin{itemize}
     \item[$\cdot$]
      Let $\vec{m}_+$ be the sub-tuple of $\vec{m}$ that consists
       of all the nonzero (hence positive) entries of $\vec{m}$;
      then $\vec{m}_+$ specifies a sub-partition $\mu=\mu_{\vec{m}}$
       of $\lambda$.
      Let
       $\pr^{\lambda}_{\mu}:
        {\Bbb P}^{\lambda}\rightarrow {\Bbb P}^{\mu}$
       be a projection map to a product of components specified by $\mu$.
      Then,
        up to an automorphism on ${\Bbb P}^{\lambda}$
         that permutes equal-dimensional factors,
       $\pr^{\lambda}_{\mu}$ is uniquely determined by $(\lambda,\mu)$
        and, hence, by $(\lambda,\vec{m})$.

     \item[$\cdot$]
      For $m_i\in \vec{m}_+\subset\vec{m}$,
       let $\hat{\pr}_i$ be the composition of projection maps
        ${\Bbb P}^{\lambda}\rightarrow {\Bbb P}^{\mu}
                                        \rightarrow {\Bbb P}^{m_i}$.
      Then,
       $L_{m_i}:=\hat{\pr}_i^{\ast}({\cal O}_{{\Bbb P}^{m_i}}(1))\,$.

     \item[$\cdot$]
      For $m_i\notin\vec{m}_+$,
      $L_{m_i}:={\cal O}_{{\Bbb P}^{\,\lambda}}\,$.
    \end{itemize}
  \end{itemize}
 Then,
  $$
   \omega_{n,1^{r_1<\,\cdots\,<r_s}}({\Bbb C})\otimes_{\Bbb Z}{\Bbb Q}\;
   =\; \bigoplus_{(\lambda,\vec{m}^{\mbox{\boldmath $\cdot$}})\,
                   \in\, {\cal P}_{n,1^{r_1<r_2<\,\cdots\,<r_s}}}
        {\Bbb Q}\cdot\phi(\lambda,\vec{m}^{\mbox{\boldmath $\cdot$}})\,.
  $$
\end{theorem}

Recall the ideal ${\frak m}\subset\omega_{\ast}({\Bbb C})$
 that is generated by elements of positive dimensions:
  $$
   0\; \longrightarrow\; {\frak m}\; \longrightarrow\;
       \omega_{\ast}({\Bbb C})\; \longrightarrow\;
       {\Bbb Z}\; \longrightarrow\; 0\,.
  $$
Let
 $$
  {\cal P}_{\ast,1^{r_1<r_2<\,\cdots\,<r_s}}\; :=\;
       \bigcup_{n=0}^{\infty}
                {\cal P}_{n,1^{r_1<r_2<\,\cdots\,<r_s}}\,.
 $$
Then, Theorem~1.1.9 says that
 $$
  {\cal B}_{\ast,1^{r_1<r_2<\,\cdots\,<r_s}}\;
   :=\; \left\{
         \phi(\lambda,\vec{m}^{\mbox{\boldmath $\cdot$}})\;:\;
          (\lambda,\vec{m}^{\mbox{\boldmath $\cdot$}})\,\in\,
                        {\cal P}_{\ast,1^{r_1<r_2<\,\cdots\,<r_s}}
        \right\}
 $$
 forms a basis of
 $\omega_{\ast,1^{r_1<\,\cdots\,<r_s}}({\Bbb C})\otimes_{\Bbb Z}{\Bbb Q}$
 as a ${\Bbb Q}$-vector space.
The set of elements in ${\cal B}_{\ast,1^{r_1<r_2<\,\cdots\,<r_s}}$
  that do not lie in the $\omega_{\ast}({\Bbb C})$-submodule
     ${\frak m}\cdot\omega_{\ast,1^{r_1<\,\cdots\,<r_s}}({\Bbb C})\,$
       of $\,\omega_{\ast,1^{r_1<\,\cdots\,<r_s}}({\Bbb C})\,$
 is given by
 $$
  {\cal B}^{\natural}_{\ast,1^{r_1<r_2<\,\cdots\,<r_s}}\;
   :=\; \left\{
         \phi(\lambda,\vec{m}^{\mbox{\boldmath $\cdot$}})\;:\;
          (\lambda,\vec{m}^{\mbox{\boldmath $\cdot$}})\,\in\,
                        {\cal P}_{\ast,1^{r_1<r_2<\,\cdots\,<r_s}}\,,\,
           l(\lambda)\le r_s\,,\, \mu_{\vec{m}}=\lambda
        \right\}
 $$
Let
 $$
  {\cal P}^{\natural}_{\ast,1^{r_1<\,\cdots\,<r_s}}\;:=\;
   \{ (\lambda,\vec{m}^{\mbox{\boldmath $\cdot$}})\,
       \in\, {\cal P}_{\ast,1^{r_1<r_2<\,\cdots\,<r_s}}\;:\;
       l(\lambda)\le r_s\,,\, \mu_{\vec{m}}=\lambda \}
 $$
 from the above characterization.
Then, one has the following two fundamental theorems
  on the structure of $\omega_{\ast,1^{r_1<\,\cdots\,<r_s}}({\Bbb C})$
  and $\omega_{\ast,1^{r_1<\,\cdots\,<r_s}}(X)$.
They follow from the construction of [Lee-P: Sec.~2] and
  an adaptation of [Lee-P: Sec.~4] to the current situation:

\begin{theorem}
{\bf [$\omega_{\ast,1^{r_1<\,\cdots\,<r_s}}({\Bbb C})$
      as $\omega_{\ast}({\Bbb C})$-module].}
{\rm (Cf.\ [Lee-P: Sec.~0.7, Theorem~2; Sec.~0.8, Theorem~3; Sec.~2].)}
 $$
  \omega_{\ast,1^{r_1<\,\cdots\,<r_s}}({\Bbb C})\;=\;
   \oplus_{(\lambda,\vec{m}^{\mbox{\boldmath $\cdot$}})
           \in{\cal P}^{\natural}_{\ast,1^{r_1<\,\cdots\,<r_s}}}
    \omega_{\ast}({\Bbb C})\,\cdot\,
     \phi(\lambda,\vec{m}^{\mbox{\boldmath $\cdot$}})\,.
 $$
\end{theorem}

\begin{theorem}
{\bf [$\omega_{\ast,1^{r_1<\,\cdots\,<r_s}}(X)$
      via $\omega_{\ast}(X)\otimes_{\omega_{\ast}(\Bbb C)}
           \omega_{\ast,1^{r_1<\,\cdots\,<r_s}}({\Bbb C})$].}
{\rm (Cf.~[Lee-P: Sec.~0.8, Theorem~3].)}
 For $X\in\boldSch_{\Bbb C}$,
 the natural map
  $$
   \gamma_X\;:\;
    \omega_{\ast}(X)\otimes_{\omega_{\ast}(\Bbb C)}
     \omega_{\ast,1^{r_1<\,\cdots\,<r_s}}({\Bbb C})\;
      \longrightarrow\;
     \omega_{\ast,1^{r_1<\,\cdots\,<r_s}}(X)
  $$
  of $\omega_{\ast}({\Bbb C})$-modules defined by
  $$
   \gamma_X
    \left(
       [Y\stackrel{f}{\rightarrow} X] \otimes
       \phi(\lambda,\vec{m}^{\mbox{\boldmath $\cdot$}})
    \right)\;
   =\; [Y\times{\Bbb P}^{\lambda}
         \stackrel{\hspace{1em}f\circ pr_Y}{
         \raisebox{.1ex}{
               \tiny$\vdash$}\!\mbox{---------}\!\!\longrightarrow}
         X \,,\,
        \pr_{{\Bbb P}^{\lambda}}^{\ast}(L_{m_1},\,\cdots\,,L_{m_r})]
  $$
  is an isomorphism of $\omega_{\ast}({\Bbb C})$-modules.
 Here,
  $(\lambda,\vec{m}^{\mbox{\boldmath $\cdot$}})\in
         {\cal P}^{\natural}_{\ast,1^{r_1<\,\cdots\,<r_s}}$ and
  $\pr_Y:Y\times{\Bbb P}^{\lambda}\rightarrow Y$,
  $\pr_{{\Bbb P}^{\lambda}}:
     Y\times{\Bbb P}^{\lambda}\rightarrow{\Bbb P}^{\lambda}$
   are the projection maps\,.
\end{theorem}

Finally, the following theorem gives a numerical aspect of
 double point cobordism classes:

\begin{theorem} {\bf [Chern invariant].}
{\rm ([Lee-P: Sec.~2.3, Sec.~2.6].)}
 Let ${\cal C}_{n,1^{r_1<\,\cdots\,<r_s}}$ be the ${\Bbb Q}$-vector space
  of graded degree $n$-polynomials in the Chern classes:
  $$
   ( c_1(T_Y), \,\cdots\,,\,c_n(T_Y);
     c_1(L_1),\,\cdots\,,c_1(L_{r_1});\,\cdots\,;
     c_1(L_{r_{s-1}+1}),\,\cdots\,,c_1(L_{r_s}) )\,.
  $$
 The pairing
  $$
   \rho\;:\;
    \omega_{n,1^{r_1<\,\cdots\,<r_s}}({\Bbb C})\otimes_{\Bbb Z}{\Bbb Q}\,
     \times\, {\cal C}_{n,1^{r_1<\,\cdots\,<r_s}}\;
    \longrightarrow\; {\Bbb Q}
  $$
  defined by
  $$
   \rho([Y, (L_1,\,\cdots\,,L_{r_s})]\,,\,\Theta)\;
    =\;\int_Y\Theta\left(
         c_1(T_Y), \,\cdots\,,\,c_n(T_Y),\,
         c_1(L_1),\,\cdots\,, c_1(L_{r_s})
                   \right)
  $$
  is nondegenerate.
 In particular, a class
  $$
   [Y, (L_1,\,\cdots\,,L_{r_1};\,\cdots\,;
        L_{r_{s-1}+1},\,\cdots\,,L_{r_s})]
   \in\omega_{n,1^{r_1<\,\cdots\,<r_s}}({\Bbb C})\otimes_{\Bbb Z}{\Bbb Q}
  $$
  is characterized by its Chern invariants.
\end{theorem}

\bigskip

\subsection{Algebraic cobordism of filtered vector bundles on varieties.}

We now explain how the construction of $\omega_{\ast,r}(X)$ in [Lee-P]
 gives immediately
 double point cobordism groups of filtered vector bundles on varieties.

\bigskip

\begin{flushleft}
{\bf Filtration, gradation, and list generalizations
     of $\omega_{\ast,r}(X)$.}
\end{flushleft}
The cobordism group $\omega_{\ast,r}(X)$ constructed in [Lee-P]
 has three natural generalizations:
 \begin{itemize}
  \item[(1)]
   Double point cobordism group $\omega_{\ast, r_1<\,\cdots\,<r_s}(X)$
    of {\it filtered vector bundles} on smooth varieties,
   constructed
    by replacing ${\cal M}_{n,1^{r_1<\,\cdots\,<r_s}}(X)$
    in the construction of $\omega_{n,1^{r_1<\,\cdots\,<r_s}}(X)$ by
   $$
    \begin{array}{l}
     {\cal M}_{n,r_1<\,\cdots\,<r_s}(X) \\[1.2ex]
     \hspace{1em} :=\;
      \left\{
       [f:Y\rightarrow X, E_1\subset\;\cdots\;\subset E_s]\:
       \left|\:
        \begin{array}{l}
         \mbox{\boldmath $\cdot$}\;\;
          \mbox{$Y\in\boldSm_{\Bbb C}$ of dimension $n$,}\\
         \mbox{\boldmath $\cdot$}\;\;
          \mbox{$f$ projective,}\\
         \mbox{\boldmath $\cdot$}\;\;
          \mbox{$E_1\subset\;\cdots\;\subset E_s$
                                     filtered vector bundle}\\
         \hspace{1em}
          \mbox{with $\rank(E_i)=r_i$}\,
        \end{array}
       \right.
      \right\}\,,\\[5ex]
      \hspace{1em}
       \mbox{for $\;\;n\in{\Bbb Z}_{\ge 0}\,$.}
    \end{array}
   $$

  \item[(2)]
   Double point cobordism group
   $\omega_{\ast, \oplus_{i=1}^sr^{\prime}_i}(X)$
   of {\it graded vector bundles} on smooth varieties,
   constructed similarly from
   $$
    \begin{array}{l}
     {\cal M}_{n,\oplus_{i=1}^sr^{\prime}_i}(X)  \\[1.2ex]
     \hspace{1em} :=\;
      \left\{
       [f:Y\rightarrow X, E_1\oplus\;\cdots\;\oplus E_s]\:
       \left|\:
        \begin{array}{l}
         \mbox{\boldmath $\cdot$}\;\;
          \mbox{$Y\in\boldSm_{\Bbb C}$ of dimension $n$,}\\
         \mbox{\boldmath $\cdot$}\;\;
          \mbox{$f$ projective,}\\
         \mbox{\boldmath $\cdot$}\;\;
          \mbox{$E_1\oplus\;\cdots\;\oplus E_s$
                                     graded vector bundle}\\
         \hspace{1em}
          \mbox{with $\grade(E_i)=i$, $\rank(E_i)=r^{\prime}_i$}\,
        \end{array}
       \right.
      \right\}\,,\\[5ex]
      \hspace{1em}
       \mbox{for $\;\;n\in{\Bbb Z}_{\ge 0}\,$.}
    \end{array}
   $$

  \item[(3)]
   Double point cobordism group
   $\omega_{\ast,(r^{\prime\prime}_i)_{i=1}^s}(X)$
   of (ordered) {\it lists of vector bundles} on smooth varieties,
   constructed similarly from
   $$
    \begin{array}{l}
     {\cal M}_{n,(r^{\prime\prime}_i)_{i=1}^s}(X)  \\[1.2ex]
     \hspace{1em} :=\;
      \left\{
       [f:Y\rightarrow X, (E_1,\;\cdots\;,E_s)]\:
       \left|\:
        \begin{array}{l}
         \mbox{\boldmath $\cdot$}\;\;
          \mbox{$Y\in\boldSm_{\Bbb C}$ of dimension $n$,}\\
         \mbox{\boldmath $\cdot$}\;\;
          \mbox{$f$ projective,}\\
         \mbox{\boldmath $\cdot$}\;\;
          \mbox{$(E_1,\;\cdots\;,E_s)$
                                 list of vector bundles}\\
         \hspace{1em}
          \mbox{with $\rank(E_i)=r^{\prime\prime}_i$}\,
        \end{array}
       \right.
      \right\}\,,\\[5ex]
      \hspace{1em}
       \mbox{for $\;\;n\in{\Bbb Z}_{\ge 0}\,$.}
    \end{array}
   $$
 \end{itemize}
Cf.~Definition~1.1.5, Definition~1.1.6, and Definition~1.1.7.

The following observation is the starting point of the current note:

\begin{lemma}
{\bf [Filtration, gradation, and list: three in one].}
 Under the above setting,
 the homomorphisms of monoids\footnote{Here,
                      we set the convention that $r_0=0$ and $E_0=0$.}
 $$
  \xymatrix{
   {\cal M}_{\ast, r_1<\,\cdots\,<r_s}(X)\ar[rr]^-{gr}
    && {\cal M}_{\ast, \oplus_{i=1}^s(r_i-r_{i-1})}(X)\ar[rr]^-{list}
    && {\cal M}_{\ast, (r_i-r_{i-1})_{i=1}^s}(X)
  }
 $$
 with
 $$
  \begin{array}{l}
   [\,f:Y\rightarrow X,
    E_1\subset E_2\subset\,\cdots\,\subset E_s\,] \\[1.2ex]
    \hspace{6em}\longmapsto\hspace{1em}
    [\,f:Y\rightarrow X,
       E_1\oplus (E_2/E_1)\oplus\,\cdots\,\oplus (E_s/E_{s-1})\,]
                                                  \\[1.2ex]
    \hspace{6em}\longmapsto\hspace{1em}
    [\,f:Y\rightarrow X,
      (E_1, E_2/E_1, \,\cdots\,, E_s/E_{s-1})\,]
   \end{array}
 $$
 induce group isomorphisms
 $$
  \xymatrix{
   \omega_{\ast, r_1<\,\cdots\,<r_s}(X)\ar[rr]^-{gr}_{\sim\hspace{1em}}
   && \omega_{\ast, \oplus_{i=1}^s(r_i-r_{i-1})}(X)\ar[rr]^-{list}_{\sim}
   && \omega_{\ast, (r_i-r_{i-1})_{i=1}^s}(X)\,,
  }
 $$
 which are also $\omega_{\ast}({\Bbb C})$-module isomorphisms.
\end{lemma}

\begin{proof} We prove the lemma in two parts.

 \bigskip
 \noindent
 $(a)$\hspace{1ex}
 The sequence of monoid homomorphisms
 $$
  \xymatrix{
   {\cal M}_{\ast, r_1<\,\cdots\,<r_s}(X)\ar[rr]^-{gr}
    && {\cal M}_{\ast, \oplus_{i=1}^s(r_i-r_{i-1})}(X)\ar[rr]^-{list}
    && {\cal M}_{\ast, (r_i-r_{i-1})_{i=1}^s}(X)
  }
 $$
 extend canonically to
 a sequence of group homomorphism after completion:
 $$
  \xymatrix{
   {\cal M}_{\ast, r_1<\,\cdots\,<r_s}(X)^+\ar[rr]^-{gr}
    && {\cal M}_{\ast, \oplus_{i=1}^s(r_i-r_{i-1})}(X)^+\ar[rr]^-{list}
    && {\cal M}_{\ast, (r_i-r_{i-1})_{i=1}^s}(X)^+\,.
  }
 $$
 The latter are compatible with double point relations over $X\,$:
 $$
  \xymatrix{
   {\cal R}_{\ast, r_1<\,\cdots\,<r_s}(X)\ar[rr]^-{gr}
    && {\cal R}_{\ast, \oplus_{i=1}^s(r_i-r_{i-1})}(X)\ar[rr]^-{list}
    && {\cal R}_{\ast, (r_i-r_{i-1})_{i=1}^s}(X)
  }
 $$
 and, hence, descend to a sequence of group homomorphisms:
 $$
  \xymatrix{
   \omega_{\ast, r_1<\,\cdots\,<r_s}(X)\ar[rr]^-{gr}
    && \omega_{\ast, \oplus_{i=1}^s(r_i-r_{i-1})}(X)\ar[rr]^-{list}
    && \omega_{\ast, (r_i-r_{i-1})_{i=1}^s}(X)\,,
  }
 $$
 which are $\omega_{\ast}({\Bbb C})$-module homomorphisms
  by construction.

 The map
  $\llist: {\cal M}_{\ast, \oplus_{i=1}^s(r_i-r_{i-1})}(X)
    \longrightarrow {\cal M}_{\ast, (r_i-r_{i-1})_{i=1}^s}(X)$
  is a monoid isomorphism.
 Its descent
  $\llist: \omega_{\ast, \oplus_{i=1}^s(r_i-r_{i-1})}(X)
           \longrightarrow \omega_{\ast, (r_i-r_{i-1})_{i=1}^s}(X)$
  is consequently a group/$\omega_{\ast}({\Bbb C})$-module isomorphism.

 The monoid homomorphism
  $\gr:{\cal M}_{\ast, r_1<\,\cdots\,<r_s}(X)
    \longrightarrow
     {\cal M}_{\ast, \oplus_{i=1}^s(r_i-r_{i-1})}(X)$
  has a right inverse, defined by
  $$
   \xymatrix{
    {\cal M}_{\ast, r_1<\,\cdots\,<r_s}(X)
     && {\cal M}_{\ast, \oplus_{i=1}^s(r_i-r_{i-1})}(X)
         \ar[ll]_-{gr^{-1,R}}
   }
  $$
  with
  $$
   \begin{array}{l}
    [\,f:Y\rightarrow X,
     V_1\subset V_1\oplus V_2 \subset\,\cdots\,
      \subset V_1\oplus V_2\oplus\,\cdots\,\oplus\,V_s\,] \\[1.2ex]
     \hspace{12em}
     \longleftarrow\!\!\mbox{---------}\!\raisebox{.1ex}{\tiny$\dashv$}
     \hspace{1em}
     [\,f:Y\rightarrow X,
        V_1\oplus V_2\oplus\,\cdots\,\oplus V_s\,]\,,
    \end{array}
  $$
  where $V_1\oplus\,\cdots\,\oplus V_i$ is contained in
   $V_1\oplus\,\cdots\,\oplus V_i\oplus V_{i+1}$ as the subbundle
   $V_1\oplus\,\cdots\,\oplus V_i\oplus{\bf 0}\,$.
 It satisfies the property that
   $\gr\circ\gr^{-1,R}$ is the identity map
   on ${\cal M}_{\ast, \oplus_{i=1}^s(r_i-r_{i-1})}(X)$.
 This $\gr^{-1,R}$ extends to
   $\gr^{-1,R}:{\cal M}_{\ast, \oplus_{i=1}^s(r_i-r_{i-1})}(X)^+
     \longrightarrow {\cal M}_{\ast, r_1<\,\cdots\,<r_s}(X)^+$
   with
    $\gr^{-1,R}({\cal R}_{\ast, \oplus_{i=1}^s(r_i-r_{i-1})}(X))
     \subset {\cal R}_{\ast, r_1<\,\cdots\,<r_s}(X)$ and, hence,
  descends to a group/$\omega_{\ast}({\Bbb C})$-module homomorphism
  $$
   \xymatrix{
    \omega_{\ast, r_1<\,\cdots\,<r_s}(X)
     && \omega_{\ast, \oplus_{i=1}^s(r_i-r_{i-1})}(X)\,.
         \ar[ll]_-{gr^{-1,R}}
   }
  $$

 \bigskip

 \noindent
 $(b)$\hspace{1ex} We claim that
 \begin{itemize}
  \item[$\cdot$] {\it
   $\gr:
    \omega_{\ast, r_1<\,\cdots\,<r_s}(X)
     \longrightarrow \omega_{\ast, \oplus_{i=1}^s(r_i-r_{i-1})}(X)$
   has the inverse $\gr^{-1}$ given by \\
   $\gr^{-1,R}:\omega_{\ast, \oplus_{i=1}^s(r_i-r_{i-1})}(X)
     \longrightarrow \omega_{\ast, r_1<\,\cdots\,<r_s}(X)$
   defined in Part (a).}
 \end{itemize}

 \medskip
 \noindent
 This is a consequence of the following two observations:

 \bigskip

 \noindent
 {\bf Lemma 1.2.1.(b.1). [from exact sequence to direct sum].}
  {\rm (Cf.\ [Lee-P: part of proof of Proposition~12].)} {\it
  Let
   $$
    0\;\longrightarrow\; E\; \longrightarrow\; F\;
       \longrightarrow\; G\; \longrightarrow\; 0
   $$
    be an exact sequence of vector bundles on $Y$
    with $F$ of rank $r$.
  Then,
   $$
    [f:Y\rightarrow X,F]\; =\; [f:Y\rightarrow X,E\oplus G] \hspace{2em}
    \mbox{in $\;\omega_{\ast,r}(X)$}\,.
   $$
 } 

 \noindent
 {\it Proof.}
  Note that
   $\Ext_Y(G,F)$ is a finite-dimensional ${\Bbb C}$-vector space
    with the origin
    $[0\,\rightarrow\, E\, \rightarrow\, E\oplus F\,
         \rightarrow\, G\; \rightarrow\, 0]\,$.
  A ${\Bbb C}$-line connecting
    $[0\,\rightarrow\, E\, \rightarrow\, F\,
         \rightarrow\, G\, \rightarrow\, 0]$ and
    $[0\,\rightarrow\, E\, \rightarrow\, E\oplus F\,
         \rightarrow\, G\; \rightarrow\, 0]\,$
   is given by the push-out (i.e.\ fibered coproduct) of
    $\,0\,\rightarrow\, E\, \rightarrow\, E\oplus F\,
          \rightarrow\, G\; \rightarrow\, 0\,$
    via the bundle homomorphism
     $E\stackrel{t\cdot}{\longrightarrow} E$,
     $v\mapsto t\cdot v$, for $t\in {\Bbb C}\,$.
  This defines
   a vector bundle $\tilde{F}_{{\Bbb A}^1}$ on $Y\times{\Bbb A}^1$,
    where ${\Bbb A}^1$ is the affine line $\Spec{\Bbb C}[t]$,
   with $\tilde{F}_t:= (\tilde{F}_{{\Bbb A}^1})|_{Y\times\{t\}}
     = F$ on $Y$ for $t\ne{\bf 0}\in{\Bbb A}^1$, and
     $=E\oplus G$ on $Y$ for $t={\bf 0}\,$.
  Explicitly,
   $\tilde{F}_{{\Bbb A}^1}$ is defined by the following exact sequence of
   vector bundles on $Y\times{\Bbb A}^1\,$:
   (for $0\,\rightarrow\, E\, \stackrel{u}{\rightarrow}\, F\,
            \rightarrow\, G\, \rightarrow\, 0$)
   $$
    \xymatrix{
    0\ar[r]
     & E_{{\Bbb A}^1}\ar[r]^-{(-u, t)}
     & F_{{\Bbb A}^1}\oplus E_{{\Bbb A}^1}\ar[r]
     & \tilde{F}_{{\Bbb A}^1}\ar[r] & 0\,,
     }
   $$
   where
    $E_{{\Bbb A}^1}$, $F_{{\Bbb A}^1}$, and $G_{{\Bbb A}^1}$
     (resp.\ $E_{{\Bbb A}^1}\stackrel{-u}{\longrightarrow}F_{{\Bbb A}^1}$)
    are the pullback of $E$, $F$, $G$
     (resp.\ $E\stackrel{-u}{\longrightarrow}F$) on $Y$
    to $Y\times {\Bbb A}^1$
    by the projection map $Y\times {\Bbb A}^1\rightarrow Y$.

 \noindent\hspace{38em}$\square$

 \bigskip

 \noindent
 {\bf Lemma~1.2.1.(b.2). [gradation cobordant to filtration].} {\it
  Let
   $\alpha\in {\cal M}_{\ast, r_1<\,\cdots\,<r_s}(X)\,$.
  Then
   $$
    [\alpha]\;=\; [\gr^{-1,R}\circ\gr(\alpha)]
     \hspace{1em}\in\; \omega_{\ast, r_1<\,\cdots\,<r_s}(X)\,.
   $$
 } 

 \noindent
 {\it Proof.}
  Let $\alpha$ be given by
   $[\,f:Y\rightarrow X, E_1\subset E_2\subset\,\cdots\,\subset E_s\,]
    \in {\cal M}_{\ast, r_1<\,\cdots\,<r_s}(X)\,$.
  Associated to $\alpha$ is an inclusion tower of short exact sequences
   of vector bundles on $Y$:
   $$
    \xymatrix{
     0 \ar[r]
      & E_1 \raisebox{-1ex}{\rule{0ex}{1ex}} \ar[r] \ar@{=}[d]
      & E_1 \raisebox{-1ex}{\rule{0ex}{1ex}} \ar[r] \ar@{^{(}->}[d]
      & 0    \raisebox{-1ex}{\rule{0ex}{1ex}} \ar[r] \ar@{^{(}->}[d]
      & 0 \\
     0 \ar[r]
      & E_1 \raisebox{-1ex}{\rule{0ex}{3ex}} \ar[r] \ar@{=}[d]
      & E_2
        \raisebox{-1ex}{\rule{0ex}{3ex}} \ar[r] \ar@{^{(}->}[d]
      & E_2/E_1
        \raisebox{-1ex}{\rule{0ex}{3ex}} \ar[r] \ar@{^{(}->}[d]
      & 0 \\
     0 \ar[r]
      & E_1 \raisebox{-1ex}{\rule{0ex}{3ex}} \ar[r] \ar@{=}[d]
      & E_3 \raisebox{-1ex}{\rule{0ex}{3ex}} \ar[r] \ar@{^{(}->}[d]
      & E_3/E_1
        \raisebox{-1ex}{\rule{0ex}{3ex}} \ar[r] \ar@{^{(}->}[d]
      & 0 \\
     0 \ar[r]
      & \hspace{.8em}\vdots\hspace{.8em}
        \raisebox{-1ex}{\rule{0ex}{3ex}} \ar[r] \ar@{=}[d]
      & \hspace{1em}\vdots\hspace{1em}
        \raisebox{-1ex}{\rule{0ex}{3ex}} \ar[r] \ar@{^{(}->}[d]
      & \hspace{1em}\vdots\hspace{1em}
        \raisebox{-1ex}{\rule{0ex}{3ex}} \ar[r] \ar@{^{(}->}[d]
      & 0 \\
     0 \ar[r]
      & E_1 \raisebox{-1ex}{\rule{0ex}{3ex}} \ar[r] \ar@{=}[d]
      & E_{s-1} \raisebox{-1ex}{\rule{0ex}{3ex}} \ar[r] \ar@{^{(}->}[d]
      & E_{s-1}/E_1 \raisebox{-1ex}{\rule{0ex}{3ex}} \ar[r] \ar@{^{(}->}[d]
      & 0 \\
     0 \ar[r]
      & E_1 \raisebox{-1ex}{\rule{0ex}{3ex}} \ar[r]
      & E_s \raisebox{-1ex}{\rule{0ex}{3ex}} \ar[r]
      & E_s/E_1 \raisebox{-1ex}{\rule{0ex}{3ex}} \ar[r]
      & 0\;,
    }
   $$
   which defines an element in the ${\Bbb C}$-vector space
    $$
     \alpha_{(1)}\;\in\;
      \prod_{i=1}^s\,\Ext_Y(E_i/E_1, E_1)\,.
    $$
   (Here, $E_0=0$ be convention.)
   Recall Lemma~1.2.1.(b.1).
   A simultaneous push-out of the above tower of exact sequences
     by the bundle homomorphism
     $E_1\stackrel{t\cdot}{\longrightarrow} E_1$,
     $v\mapsto t\cdot v$, for $t\in {\Bbb C}$
    gives rise to a filtered vector bundles
     $$
      F^{\bullet}E^{\,\sim}_{s, {\Bbb A}^1}\;\;:\;\;
       E^{\,\sim}_{1,{\Bbb A}^1}\; \subset\; E^{\,\sim}_{2,{\Bbb A}^1}\;
       \subset\; \cdots\; \subset\; E^{\,\sim}_{s,{\Bbb A}^1}
     $$
     on $Y\times{\Bbb A}^1$ over ${\Bbb A}^1=\Spec{\Bbb C}[t]$
    with
     $(F^{\bullet}E^{\,\sim}_{s,{\Bbb A}^1})_t
      := (F^{\bullet}E^{\,\sim}_{s, {\Bbb A}^1})|_{Y\times\{t\}}
       = (E_1\subset E_2\subset\,\cdots\,\subset E_s)$
       on $Y$ at $t\ne{\bf 0}\in{\Bbb A}^1$, and
     $=\,$ the direct sum
      $$
       F^{\bullet}(E_1\oplus(E_s/E_1))\;\;:\;\;
        E_1\;\subset\; E_1\oplus (E_2/E_1)\;\subset\;\cdots\;\subset\;
        E_1\oplus(E_s/E_1)
      $$
      of the first and the third column of the above tower,
      as filtered vector bundles, on $Y$ at $t={\bf 0}\,$.

   Applying the push-out construction next to the filtered vector bundle
    $$
     F^{\bullet}(E_s/E_1)\;\;:\;\;
      E_2/E_1\;\subset\; E_3/E_1\;\subset\; \cdots\; \subset\; E_s/E_1
    $$
     with the canonical isomorphisms
     $(E_i/E_j)/(E_{i^{\prime}}/E_j)\simeq E_i/E_{i^{\prime}}$
     taken into account,
    augmenting the result on the left by $0\subset\,\bullet\,$,
     and then
    taking a direct sum with the pull-back of the constant filtration
     $E_1\subset\,\cdots\,\subset E_1$ (with $s$-many terms) on $Y$
     to $Y\times{\Bbb A}^1$ via the projection map
     $Y\times{\Bbb A}^1\rightarrow Y$,
   one obtains a second filtered vector bundle
     $$
       F^{\bullet}(E_1\oplus(E_s/E_1))
                               ^{\,\sim}_{{\Bbb A}^1}\;\;:\;\;
       E^{\,\sim}_{1,{\Bbb A}^1}\; \subset\;
       (E_1\oplus(E_2/E_1))^{\,\sim}_{{\Bbb A}^1}\;
       \subset\; \cdots\; \subset\;
       (E_1\oplus(E_s/E_1))^{\,\sim}_{{\Bbb A}^1}
     $$
     on $Y\times{\Bbb A}^1$ over ${\Bbb A}^1=\Spec{\Bbb C}[t]$
    with
     $(F^{\bullet}(E_1\oplus(E_s/E_1))^{\,\sim}_{{\Bbb A}^1})_t
      := (F^{\bullet}(E_1\oplus(E_s/E_1))
                     ^{\,\sim}_{{\Bbb A}^1})|_{Y\times\{t\}}\,
       = F^{\bullet}(E_1\oplus(E_s/E_1))$, as constructed previously,
       on $Y$ at $t\ne{\bf 0}\in{\Bbb A}^1$, and equal to
      $$
       \begin{array}{l}
       F^{\bullet}(E_1\oplus(E_2/E_1)\oplus(E_s/E_2))\;\;:\;\;\\[1.2ex]
       \hspace{2em}
        E_1\;\subset\; E_1\oplus (E_2/E_1)\;\subset\;
        E_1\oplus(E_2/E_1)\oplus(E_3/E_2)\;
        \subset\;  \cdots\; \subset\;
        E_1\oplus(E_2/E_1)\oplus(E_{s}/E_2)
       \end{array}
      $$
      on $Y$ at $t={\bf 0}\,$.

   Repeating this procedure of
    \begin{itemize}
     \item[$\cdot$]
      {\it push-out $\Rightarrow$
        augment on the left by a $0$ filtration $\Rightarrow$
        direct sum with a constant filtration},
    \end{itemize}
   one obtains a length-$(s-1)$ sequence of filtered vector bundles
    on $Y\times{\Bbb A}^1$ that interpolate consecutively
   $$
    F^{\bullet}E_s\; \leadsto\;
    F^{\bullet}(E_1\oplus(E_s/E_1))\;\leadsto\;
    F^{\bullet}(E_1\oplus(E_2/E_1)\oplus(E_s/E_2))\;
     \leadsto\;\cdots\;\leadsto\;
    F^{\bullet}(\oplus_{i=1}^s(E_i/E_{i-1}))\,.
   $$
  It follows that
   \begin{eqnarray*}
    [\alpha]
     & = & [f:Y\rightarrow X, F^{\bullet}(E_1\oplus(E_s/E_1))] \\[1.2ex]
     & = & \cdots\;
      =\; [f:Y\rightarrow X, F^{\bullet}(\oplus_{i=1}^s(E_i/E_{i-1}))]\;
      =\; [\gr^{-1,R}\circ\gr(\alpha)]
   \end{eqnarray*}
   in $\omega_{\ast, r_1<\,\cdots\,<r_s}(X)\,$.
  This proves the sub-lemma.

 \noindent\hspace{38em}$\square$

\end{proof}

\bigskip

\begin{flushleft}
{\bf $\omega_{\ast, (r_i-r_{i-1})_{i=1}^s}(X)$ as a quotient of
     $\omega_{\ast, 1^{r_1<\,\cdots\,<r_s}}(X)$.}
\end{flushleft}
Through the canonical isomorphism
 $\omega_{\ast, r_1<\,\cdots\,<r_s}(X)
   \simeq \omega_{\ast, (r_i-r_{i-1})_{i=1}^s}(X)$
 in  Lemma~1.2.1,
 one can shift the focus
  from the original cobordism groups of filtered bundles on varieties
  to the (equivalent but more convenient) cobordism groups of
  lists of vector bundles on varieties.

The canonical monoid homomorphism
 $$
  \kappa_{\cal M}\; :\;
   {\cal M}_{\ast, 1^{r_1<\,\cdots\,<r_s}}(X)\;
    \longrightarrow\; {\cal M}_{\ast, (r_i-r_{i-1})_{i=1}^s}(X)
 $$
 with
 $$
  \begin{array}{l}
   [\,f:Y\rightarrow X,
       (L_1,\,\cdots\,,L_{r_1}; L_{r_1+1},\,\cdots\,,L_{r_2};
        \,\cdots\,; L_{r_{s-1}+1},\,\cdots\,,L_{r_s})\,]    \\[1.2ex]
   \hspace{8em}\longmapsto\;
   \left[
    f:Y\rightarrow X\,,\,
       \left(
        \oplus_{i=1}^{r_1}L_i\,,\, \oplus_{i=r_1+1}^{r_2}L_i\,,\,
         \cdots\,,\,\oplus_{i=r_{s-1}+1}^{r_s}L_i
        \right)
    \right]
  \end{array}
 $$
 extends to a group homomorphism
 $$
  \kappa_{{\cal M}^+}\; :\;
   {\cal M}_{\ast, 1^{r_1<\,\cdots\,<r_s}}(X)^+\;
    \longrightarrow\; {\cal M}_{\ast, (r_i-r_{i-1})_{i=1}^s}(X)^+
 $$
 that maps ${\cal R}_{\ast, 1^{r_1<\,\cdots\,<r_s}}(X)$
  to ${\cal R}_{\ast, (r_i-r_{i-1})_{i=1}^s}(X)$
  and, hence,
 descends to a group/$\omega_{\ast}({\Bbb C})$-module homomorphism
 $$
  \kappa := \kappa_{\omega}\; :\;
   \omega_{\ast, 1^{r_1<\,\cdots\,<r_s}}(X)\;
    \longrightarrow\; \omega_{\ast, (r_i-r_{i-1})_{i=1}^s}(X)\,.
 $$
 The permutation group $\times_{i=1}^s\Sym_{\,r_i-r_{i-1}}$ acts
  on ${\cal M}_{\ast, 1^{r_1<\,\cdots\,<r_s}}(X)$
  by permutation of the elements of a list in a list of lists:
  $$
   \begin{array}{l}
    [\,f:Y\rightarrow X,
        (L_1,\,\cdots\,,L_{r_1}; L_{r_1+1},\,\cdots\,,L_{r_2};
         \,\cdots\,; L_{r_{s-1}+1},\,\cdots\,,L_{r_s})\,]\;\;
    \stackrel{\sigma=(\sigma_1,\,\cdots\,\sigma_s)}{
     \raisebox{.1ex}{
      \tiny$\vdash$}\!\mbox{---------------}\!\!\longrightarrow
                     }                                   \\[1.2ex]
    \hspace{4em}
        [\,f:Y\rightarrow X,
        (L_{\sigma_1(1)},\,\cdots\,,L_{\sigma_1(r_1)};
         L_{\sigma_2(r_1+1)},\,\cdots\,,L_{\sigma_2(r_2)};
         \,\cdots\,;
         L_{\sigma_s(r_{s-1}+1)},\,\cdots\,,L_{\sigma_s(r_s)})\,]\,.
   \end{array}
  $$
 This extends to an action on the group completion
  ${\cal M}_{\ast, (r_i-r_{i-1})_{i=1}^s}(X)^+$,
  which, then, descends to an action on its quotient
  $\omega_{\ast, (r_i-r_{i-1})_{i=1}^s}(X)$.
 Endow
  ${\cal M}_{\ast, (r_i-r_{i-1})_{i=1}^s}(X)$,
  ${\cal M}_{\ast, (r_i-r_{i-1})_{i=1}^s}(X)^+$, and
  $\omega_{\ast, (r_i-r_{i-1})_{i=1}^s}(X)$
  with the trivial $\times_{i=1}^s\Sym_{\,r_i-r_{i-1}}$-action.
 Then $\kappa_{\cal M}$, $\kappa_{{\cal M}^+}$, and $\kappa$
  are all $\times_{i=1}^s\Sym_{\,r_i-r_{i-1}}$-equivariant maps.

\begin{proposition}
{\bf [$\omega_{\ast, (r_i-r_{i-1})_{i=1}^s}(X)$
      as quotient of $\omega_{\ast, 1^{r_1<\,\cdots\,<r_s}}(X)$].}
 The $\times_{i=1}^s\Sym_{\,r_i-r_{i-1}}$-equivariant map
  $\kappa\,:\,
   \omega_{\ast, 1^{r_1<\,\cdots\,<r_s}}(X)\,
    \longrightarrow\,\omega_{\ast, (r_i-r_{i-1})_{i=1}^s}(X)\,$
  constructed above
 is a group/$\omega_{\ast}({\Bbb C})$-module epimorphism.
\end{proposition}

\begin{proof}
 This follows from the same construction as in [Lee-P: Sec.~3.1].
 It is enough to check that any
  $\alpha$ in $\omega_{\ast, (r_i-r_{i-1})_{i=1}^s}(X)$ of the form
  $[f:Y\rightarrow X, (V_1,\,\cdots\,,V_s)]$ lies in $\Image\kappa$.
 First, note that if $\dimm Y=0$, then $\alpha\in \Image\kappa\,$
  since any ${\Bbb C}$-vector space splits.
 In general,
  let ${\frak m}\subset\omega_{\ast}({\Bbb C})$ be the ideal
  generated by elements of positive dimensions:
  $$
   0\; \longrightarrow\; {\frak m}\; \longrightarrow\;
       \omega_{\ast}({\Bbb C})\; \longrightarrow\;
       {\Bbb Z}\; \longrightarrow\; 0\,.
  $$
 Then, there exists a {\it smooth} variety $\hat{Y}$
  with a birational morphism $\hat{Y}\rightarrow Y$
  such that\footnote{This step
                     requires the result of Hironaka [Hi]
                     on resolution of singularities.}
  \begin{itemize}
   \item[$\cdot$]
    $(V_1,\,\cdots\,,V_s)$ becomes a list of splittable vector bundles
    after being pulled back to $\hat{Y}$;

   \item[$\cdot$]
    denote
     the composition $\hat{Y}\rightarrow Y\stackrel{f}{\rightarrow} X$
      by $\hat{f}$  and
     the pull-back of $(V_1,\,\cdots\,,V_s)$ to $\hat{Y}$
      by $(\hat{V}_1,\,\cdots\,,\hat{V}_s)$,
    then
     $$
      [\hat{f}:\hat{Y}\rightarrow X, (\hat{V}_1,\,\cdots\,,\hat{V}_s)]\,
       -\, [f:Y\rightarrow X, (V_1,\,\cdots\,,V_s)]\;
       \in\; {\frak m}\cdot\omega_{\ast, (r_i-r_{i-1})_{i=1}^s}(X)\,.
     $$
  \end{itemize}
 The proposition now follows by an induction on $\dimm Y$.

\end{proof}


\bigskip

\begin{flushleft}
{\bf Algebraic cobordism of filtered vector bundles on varieties}
\end{flushleft}
With
 the preparation in the previous two themes of this subsection,
 [Lee-P],  and
 Sec.~1.1 of the current note,
one has now the following immediate generalizations of
 Lee and Pandharipande
 to algebraic cobordisms of filtered vector bundles on varieties
 via $\kappa\,$:

\begin{theorem}
{\bf [$\omega_{n,r_1<\,\cdots\,<r_s}({\Bbb C})\otimes_{\Bbb Z}{\Bbb Q}$
      as ${\Bbb Q}$-vector space].}
{\rm (Cf.\ [Lee-P: Theorem~1].)}
 {To} each
  $(\lambda,\vec{\mu})=(\lambda,(\mu_1,\,\cdots\,,\mu_s))
    \in {\cal P}_{n,r_1<r_2<\,\cdots\,<r_s}$,
  associate an element
  $$
   \phi(\lambda,\vec{\mu})\;=\;
   [{\Bbb P}^{\lambda}\,,\, E_1\subset\,\cdots\,\subset E_s]
  $$
  in $\omega_{n,r_1<r_2<\,\cdots\,<r_s}({\Bbb C})$,
 where
  \begin{itemize}
   \item[$\cdot$]
    ${\Bbb P}^{\,\lambda}
     := {\Bbb P}^{\,\lambda_1}\times\,
        \cdots\,\times{\Bbb P}^{\,\lambda_{l(\lambda)}}\,$
    is a product of projective spaces specified by $\lambda$,

   \item[$\cdot$]
    $E_1\subset\,\cdots\,\subset E_s$ is a filtered vector bundle
    on ${\Bbb P}^{\,\lambda}$ defined as follows:
    \begin{itemize}
     \item[$\cdot$]
      Let
       $\mu\subset \lambda$ be the union of $\mu_1,\,\cdots\,,\mu_s$ and
       $\pr^{\lambda}_{\mu}:
        {\Bbb P}^{\lambda}\rightarrow {\Bbb P}^{\mu}$
       be a projection map to a product of components of
        ${\Bbb P}^{\lambda}$ specified by $\mu$.
      Then,
        up to an automorphism on ${\Bbb P}^{\lambda}$
         that permutes equal-dimensional factors,
       $\pr^{\lambda}_{\mu}$ is uniquely determined by $(\lambda,\mu)$
        and, hence, by $(\lambda,\vec{\mu})$.

     \item[$\cdot$]
      For $m\in\mu$,
       let $\hat{\pr}_{(m)}$ be the composition of projection maps
        ${\Bbb P}^{\lambda}\rightarrow {\Bbb P}^{\mu}
                                        \rightarrow {\Bbb P}^m$ and
      define the line bundle
       $L_m:=\hat{\pr}_{(m)}^{\ast}({\cal O}_{{\Bbb P}^m}(1))$
       on ${\Bbb P}^{\lambda}$ via pull-back.

     \item[$\cdot$]
      In terms of these line bundles,
      $$
       \begin{array}{crl}
        E_i  & :=
         & {\cal O}_{{\Bbb P}^{\lambda}}^{r_i-\sum_{j=1}^il(\mu_j)}
            \oplus
            \bigoplus_{m\,\in\,\mu_1\cup\,\cdots\,\cup\,\mu_i} L_m
                                                               \\[1.2ex]
         & = & {\cal O}_{{\Bbb P}^{\lambda}}^{r_i-\sum_{j=1}^il(\mu_j)}
                \oplus
                \bigoplus_{m\,\in\,\mu_1\cup\,\cdots\,\cup\,\mu_i}
                 \hat{\pr}_{(m)}^{\ast}({\cal O}_{{\Bbb P}^m}(1))\,.
       \end{array}
      $$
    \end{itemize}
  \end{itemize}
 Then,
  $$
   \omega_{n,r_1<\,\cdots\,<r_s}({\Bbb C})\otimes_{\Bbb Z}{\Bbb Q}\;
   =\; \bigoplus_{(\lambda,\vec{\mu})\,
                   \in\, {\cal P}_{n,r_1<r_2<\,\cdots\,<r_s}}
        {\Bbb Q}\cdot\phi(\lambda,\vec{\mu})\,.
  $$
\end{theorem}

Recall the ideal ${\frak m}\subset\omega_{\ast}({\Bbb C})$.
Let
 $$
  {\cal P}_{\ast,r_1<r_2<\,\cdots\,<r_s}\; :=\;
       \bigcup_{n=0}^{\infty} {\cal P}_{n,r_1<r_2<\,\cdots\,<r_s}\,.
 $$
Then, Theorem~1.2.3 says that, as a ${\Bbb Q}$-vector space,
 $\omega_{\ast,r_1<\,\cdots\,<r_s}({\Bbb C})\otimes_{\Bbb Z}{\Bbb Q}\,$
  has a basis
 $$
  {\cal B}_{\ast,r_1<r_2<\,\cdots\,<r_s}\;
   :=\; \left\{
         \phi(\lambda,\vec{\mu})\;:\;
          (\lambda,\vec{\mu})\,\in\,
                        {\cal P}_{\ast,r_1<r_2<\,\cdots\,<r_s}
        \right\}\,.
 $$
The set of elements in ${\cal B}_{\ast,r_1<r_2<\,\cdots\,<r_s}$
  that do not lie in the $\omega_{\ast}({\Bbb C})$-submodule
     ${\frak m}\cdot\omega_{\ast,r_1<\,\cdots\,<r_s}({\Bbb C})\,$
       of $\,\omega_{\ast,r_1<\,\cdots\,<r_s}({\Bbb C})\,$
 is given by
 $$
  {\cal B}^{\natural}_{\ast,r_1<r_2<\,\cdots\,<r_s}\;
   :=\; \left\{
         \phi(\lambda,\vec{\mu})\;:\;
          (\lambda,\vec{\mu})\,\in\,
                        {\cal P}_{\ast,r_1<r_2<\,\cdots\,<r_s}\,,\,
           l(\lambda)\le r_s\,,\, \mu:=\cup_{i=1}^s\mu_i=\lambda
        \right\}
 $$
Let
 $$
  {\cal P}^{\natural}_{\ast,r_1<\,\cdots\,<r_s}\;:=\;
   \{ (\lambda,\vec{\mu})\,
       \in\, {\cal P}_{\ast,r_1<r_2<\,\cdots\,<r_s}\;:\;
       l(\lambda)\le r_s\,,\, \mu=\lambda\}
 $$
 from the above characterization.
Then, one has the following two fundamental theorems
  on the structure of $\omega_{\ast,r_1<\,\cdots\,<r_s}({\Bbb C})$
  and $\omega_{\ast,r_1<\,\cdots\,<r_s}(X)$.

\begin{theorem}
{\bf [$\omega_{\ast,r_1<\,\cdots\,<r_s}({\Bbb C})$
      as $\omega_{\ast}({\Bbb C})$-module].}
{\rm (Cf.\ [Lee-P: Theorem~2].)}
 $\omega_{\ast,r_1<\,\cdots\,<r_s}({\Bbb C})$
  is a free $\omega_{\ast}({\Bbb C})$-module:
 $$
  \omega_{\ast,r_1<\,\cdots\,<r_s}({\Bbb C})\; =\;
   \bigoplus_{(\lambda,\vec{\mu})
           \in{\cal P}^{\natural}_{\ast,r_1<\,\cdots\,<r_s}}
    \omega_{\ast}({\Bbb C})\,\cdot\, \phi(\lambda,\vec{\mu})\,.
 $$
\end{theorem}

\begin{theorem}
{\bf [$\omega_{\ast,1^{r_1<\,\cdots\,<r_s}}(X)$
      via $\omega_{\ast}(X)\otimes_{\omega_{\ast}(\Bbb C)}
           \omega_{\ast,1^{r_1<\,\cdots\,<r_s}}({\Bbb C})$].}
{\rm (Cf.\ [Lee-P: Theorem~3].)}
 For $X\in\boldSch_{\Bbb C}$,
 the natural map
  $$
   \gamma_X\;:\;
    \omega_{\ast}(X)\otimes_{\omega_{\ast}(\Bbb C)}
     \omega_{\ast,r_1<\,\cdots\,<r_s}({\Bbb C})\;
      \longrightarrow\;
     \omega_{\ast,r_1<\,\cdots\,<r_s}(X)
  $$
  of $\omega_{\ast}({\Bbb C})$-modules defined by
  $$
   \gamma_X
    \left(
       [Y\stackrel{f}{\rightarrow} X] \otimes
       \phi(\lambda,\vec{\mu})
    \right)\;
    =\;[Y\times{\Bbb P}^{\lambda}
          \stackrel{\hspace{1em}f\circ pr_Y}{
          \raisebox{.1ex}{
                \tiny$\vdash$}\!\mbox{---------}\!\!\longrightarrow}
         X \,,\,
        \pr_{{\Bbb P}^{\lambda}}^{\ast}
                               (E_1\subset \,\cdots\,\subset E_s)]
  $$
  is an isomorphism of $\omega_{\ast}({\Bbb C})$-modules.
 Here,
  $(\lambda,\vec{\mu})\in
         {\cal P}^{\natural}_{\ast,r_1<\,\cdots\,<r_s}$;
  $\pr_Y:Y\times{\Bbb P}^{\lambda}\rightarrow Y$,
  $\pr_{{\Bbb P}^{\lambda}}:
     Y\times{\Bbb P}^{\lambda}\rightarrow{\Bbb P}^{\lambda}$
   are the projection maps; and
  $E_1\subset \,\cdots\,\subset E_s$ is the filtered vector bundle
   on ${\Bbb P}^{\lambda}$ associated to $(\lambda,\vec{\mu})$
   constructed in Theorem~1.2.3.
\end{theorem}

Finally, one has also the following numerical aspect of double point
 cobordism classes for filtered vector bundles on varieties:

\begin{theorem}
{\bf [Chern invariant].} {\rm (Cf.\ [Lee-P: Theorem 4].)}
 Let ${\cal C}_{n,r_1<\,\cdots\,<r_s}$ be the ${\Bbb Q}$-vector space
  of graded degree $n$-polynomials in the Chern classes:
  $$
   \begin{array}{l}
   ( c_1(T_Y), \,\cdots\,,\,c_n(T_Y)\,;\,
     c_1(E_1),\,\cdots\,,c_{r_1}(E_1)\,;\,
     c_1(E_2/E_1),\,\cdots\,,c_{r_2-r_1}(E_2/E_1)\,;\,  \\[1.2ex]
   \hspace{16em}
     \cdots\cdots\cdots\,;
     c_1(E_s/E_{s-1}),\,\cdots\,,c_{r_s-r_{s-1}}(E_s/E_{s-1}) )\,.
   \end{array}
  $$
 The pairing
  $$
   \rho\;:\;
    \omega_{n,r_1<\,\cdots\,<r_s}({\Bbb C})\otimes_{\Bbb Z}{\Bbb Q}\,
     \times\, {\cal C}_{n,r_1<\,\cdots\,<r_s}\;
    \longrightarrow\; {\Bbb Q}
  $$
  defined by
  \begin{eqnarray*}
   \lefteqn{\rho([Y, E_1\subset\,\cdots\,\subset E_s]\,,\,\Theta)}\\[1.2ex]
    && =\;\int_Y\Theta
            ( c_1(T_Y), \,\cdots\,,\,c_n(T_Y)\,;\,
              c_1(E_1),\,\cdots\,,c_{r_1}(E_1)\,;\,
              c_1(E_2/E_1),\,\cdots\,,c_{r_2-r_1}(E_2/E_1)\,;\,   \\[1.2ex]
    && \hspace{16em}
       \cdots\cdots\cdots\,;
       c_1(E_s/E_{s-1}),\,\cdots\,,c_{r_s-r_{s-1}}(E_s/E_{s-1}) )
  \end{eqnarray*}
  is nondegenerate.
 In particular, a class
  $$
   [Y, E_1\subset\,\cdots\,\subset E_s]
   \in\omega_{n,r_1<\,\cdots\,<r_s}({\Bbb C})\otimes_{\Bbb Z}{\Bbb Q}
  $$
  is characterized by its Chern invariants.
\end{theorem}

\bigskip

\begin{flushleft}
{\bf $\omega_{\ast,\,F^{\bullet}\ast}$ as a functor.}
\end{flushleft}
Part of the discussions in this subsection can be summarized
 in the following categorical picture.

\begin{definition}
{\bf [category ${\Bbb N}^{\,F^{\bullet}}$
      of filtered positive integers].} {\rm
 The category ${\Bbb N}^{\,F^{\bullet}}$
  of {\it filtered positive integers} is defined as follows:
  \begin{itemize}
   \item[$\cdot$]
    {\it Objects.}\hspace{1em}
    Finite (strictly) increasing sequence
     $r_1<\,\cdots\,<r_s$ of positive integers.
    We'll call them {\it filtered positive integers}.\footnote{Here,
                         we think of $r_1<\,\cdots\,<r_s$
                         as a ``filtration" of $r_s$
                         with respect to `$<$'.
                        For convenience,
                         when a filtration $r_1<\,\cdots\,<r_s$
                         of $r_s$ is implicitly specified,
                        we will denote $r_1<\,\cdots\,<r_s$
                         simply by $F^{\bullet}r_s\,$.}
    (The length $s$ of
     $(r_1<\,\cdots\,<r_s)\in {\Bbb N}^{\,F^{\bullet}}$ can vary.)
         
   \item[$\cdot$]
    {\it Morphisms.}\hspace{1em}
    An arrow
     $$
      (r_1<\,\cdots\,<r_s)\;\longrightarrow\;
      (r^{\prime}_1<\,\cdots\,<r^{\prime}_{s^{\prime}})
     $$
      is given
     if and only if
      \begin{itemize}
       \item[$\cdot$]
        $r_{s}\,=\,r^{\prime}_{s^{\prime}}$, and

       \item[$\cdot$]
        $r^{\prime}_1<\,\cdots\,<r^{\prime}_{s^{\prime}}$
        is obtained from $r_1<\,\cdots\,<r_s$
        by removing some (possibly none) of the terms ($\ne r_s$)
         in the filtration.
      \end{itemize}
     Note that the set of morphisms thus defined
      satisfies the composition law.
  \end{itemize}
 Note that ${\Bbb N}\subset {\Bbb N}^{\,F^{\bullet}}$ as a subcategory,
  objects of which will be called {\it unfiltered} positive integers.
 A filtration $r_1<\,\cdots\,<r_s$ is said to be {\it complete}
  if it is of the form $1<2<\,\cdots\,<s-1<s$
  (i.e.\ $r_i=i$ for all $i$).
 The category ${\Bbb N}^{\,F^{\bullet}}$ has
  {\it initial elements} (resp.\ {\it terminal elements})\footnote{Here,
                        an object $x$ in a category ${\cal C}$ is called
                        an {\it initial} (resp.\ {\it terminal}) element 
                       if there exist no morphisms of the form
                        $y\rightarrow x$ (resp.\ $x\rightarrow y$)
                        with $y\ne x$.}
  given exactly by the set of completely filtered positive integers
  (resp.\ unfiltered positive integers).
}\end{definition}

Consider the product category\footnote{Here,
                       the product ${\cal C}_1\times{\cal C}_2$
                        of categories ${\cal C}_1$ and ${\cal C}_2$
                        is defined to be
                       the category
                        whose set of objects is given by
                         $\mbox{\it Obj}\,({\cal C}_1)\times
                          \mbox{\it Obj}\,({\cal C}_2)$.
                       Given
                        $(x_1,y_1)$ and $(x_2,y_2)
                          \in {\cal C}_1\times{\cal C}_2$,
                       the set of morphism 
                        $\mbox{\it Mor}\,((x_1,y_1),(x_2,y_2))$
                        is defined to be
                        $\mbox{\it Mor}\,(x_1,x_2)\times
                         \mbox{\it Mor}\,(y_1,y_2)$.}
 $\boldSch_{{\Bbb C}}\times {\Bbb N}^{\,F^{\bullet}}$ and
let
 $\categoryAb^{\,gr}$
  be the category of graded abelian groups  and
 $\ModCategory_{\omega_{\ast}({\Bbb C})}$ be the category
  of (left) $\omega_{\ast}({\Bbb C})$-modules.

\begin{theorem}
{\bf [$\omega_{\ast,\,F^{\bullet}\ast}$ as a functor].}
 The double point cobordism defines a covariant functor
  $$
   \begin{array}{cccclc}
    \omega_{\ast,\,F^{\bullet}\ast}  & :
     & \boldSch_{{\Bbb C}}\times {\Bbb N}^{\,F^{\bullet}}
     & \longrightarrow  & \hspace{1.6em}
                          \ModCategory_{\omega_{\ast}({\Bbb C})}\;\;
                          \subset\;\;\categoryAb^{\,gr}     \\[1.2ex]
    &&  (X,F^{\bullet}r)   &  \longmapsto
     & \omega_{\ast,\,F^{\bullet}r}(X)   & .
   \end{array}
  $$
 It has the property that,
  when restricted to an $X\in\boldSch_{\Bbb C}$:
  $$
   \begin{array}{cccclc}
    \omega_{\ast,\,F^{\bullet}\ast}(X,\,\cdot\,)  & :
     & {\Bbb N}^{\,F^{\bullet}}
     & \longrightarrow  & \hspace{1.6em}
                          \ModCategory_{\omega_{\ast}({\Bbb C})}\;\;
                          \subset\;\;\categoryAb^{\,gr}     \\[1.2ex]
    &&  F^{\bullet}r   &  \longmapsto
     & \omega_{\ast,\,F^{\bullet}r}(X)   & \,,
   \end{array}
  $$
  $\omega_{\ast,\,F^{\bullet}\ast}(X,\,\cdot\,)$ sends
   an initial element (resp.\ terminal element)
     in ${\Bbb N}^{\,F^{\bullet}}$
    to a double point cobordism group
     $\omega_{\ast,1<2<\,\cdots<r-1<r}(X)\simeq\omega_{\ast, 1^r}(X)$
     (resp.\ $\omega_{\ast,r}(X)$),
      for some $r$,
     of completely filtered vector bundles on varieties
     (resp.\ vector bundles on varieties),  and
   a morphism in ${\Bbb N}^{\,F^{\bullet}}$ to an epimorphism in
    $\ModCategory_{\omega_{\ast}({\Bbb C})}\subset\categoryAb^{\,gr}\,$.
\end{theorem}

\bigskip

\section{Algebraic cobordism of and
    basic operations on vector\\ bundles on varieties.}

We address in this section
 the issue of compatibility between
  basic operations on vector bundles
  and algebraic cobordisms of vector bundles on varieties.
This brings out also a question that we pose in the end.

\bigskip

\begin{flushleft}
{\bf (Non)compatibility between basic operations and algebraic cobordisms.}
\end{flushleft}
Recall the following theorem of Lee-Pandharipande:
(Cf.\ Theorem~1.2.6 with $s=1$.)

\begin{stheorem}
{\bf [Chern invariant].} {\rm ([Lee-P: Theorem~4].)}
 The Chern invariants respect algebraic cobordism.
 The resulting map
  $$
    \omega_{n,r}(k)\otimes_{\Bbb Z}{\Bbb Q}\;\longrightarrow\;
    {\cal C}_{n,r}^{\,\ast}
  $$
  is an isomorphism.
\end{stheorem}

We'll use this to show,
 by an example on vector bundles on Calabi-Yau $3$-folds, that:
 \begin{itemize}
  \item[$\cdot$] {\it
   Among the four basic operations
    -- direct sum, tensor product, dualization, and {\it ${\cal H}$om} --
    on vector bundles on varieties,
   only dualization is compatible with algebraic cobordisms
    of vector bundles on varieties in general.}
 \end{itemize}

\begin{sproposition}
{\bf [compatibility with dualization].}
 The automorphism $\tau$ on the monoid ${\cal M}_{\ast,r}(X)$
  that sends $(f:Y\rightarrow X, E)$ to $(f:Y\rightarrow X, E^{\vee})$
  induces an involution\footnote{I.e.\
                       an automorphism $\tau$ of the group/module
                       such that $\tau\circ\tau$ is the identity map.}
   on the group/$\omega_{\ast}({\Bbb C})$-module
  $\omega_{\ast,r}(X)$.
 In particular,
 if $\sum_i a_i[f_i:Y_i\rightarrow X,E_i]
     =\sum_j b_j[g_j:Z_j\rightarrow X, F_j]$
  in $\omega_{\ast,r}(X)$,
 then
  $\sum_i a_i[f_i:Y_i\rightarrow X,E_i^{\vee}]
   =\sum_j b_j[g_j:Z_j\rightarrow X, F_j^{\vee}]$ in $\omega_{\ast,r}(X)$.
\end{sproposition}

\begin{proof}
 $\tau$ extends naturally to an involution on the group
  ${\cal M}_{\ast,r}(X)^+$.
 Since dualization $(\,\cdot\,)^{\vee}$ of vector bundles commutes
  both with pullbacks and with restrictions of vector bundles,
 $\tau$ leaves the subgroup $R_{\ast,r}(X)\subset {\cal M}_{\ast,r}(X)^+$
  generated by all double point relations over $X$ invariant.
 The proposition thus follows.

\end{proof}

\begin{sexample}
{\bf [Noncompatibility with direct sum, tensor product,
      and {\it ${\cal H}$om}].}
{\rm
 Let
  \begin{itemize}
   \item[$\cdot$]
    $X$ be a (smooth) complete intersection Calabi-Yau $3$-fold
    from the zero-locus of a section of the rank-$2$ vector bundle
    ${\cal O}(2,2,0)\oplus {\cal O}(1,1,2)$
    on ${\Bbb P}^2\times{\Bbb P}^2\times {\Bbb P}^1$.
  \end{itemize}
 We will give examples of noncompatibility of algebraic cobordism
  with direct sum, tensor product, and {\it ${\cal H}$om}
  in Parts (d), (e), and (f)
  after the preparation in Parts (a), (b), and (c).

 \bigskip

 \noindent
 {$(a)$} {\bf Lemma [basic identity].} {\it Under the situation,
  $$
   [X\hookrightarrow {\Bbb P}^2\times{\Bbb P}^2\times{\Bbb P}^1,
    {\cal O}(d_1,d_2,d_3)|_X]\;
    =\; [X\hookrightarrow {\Bbb P}^2\times{\Bbb P}^2\times{\Bbb P}^1,
         {\cal O}(d_2,d_1,d_3)|_X]
  $$
   in $\omega_{3,1}({\Bbb P}^2\times{\Bbb P}^2\times{\Bbb P}^1)$
   and, hence, also in $\omega_{3,1}({\Bbb C})$ after pushing forward.
 } 

 \begin{proof}
  By Axiom (Sect) and Axiom (FGL)
   ([L-M: Definition~2.2.1]),\footnote{To avoid a notation jam,
                 here we denote the class
                 $[X^{\prime}\hookrightarrow
                   {\Bbb P}^2\times{\Bbb P}^2\times{\Bbb P}^1,
                   {\cal O}(d_1,d_2,d_3)|_{X^{\prime}}]
                   \in \Omega_{\ast}(
                        {\Bbb P}^2\times{\Bbb P}^2\times{\Bbb P}^1)$
                also by
                 $[X^{\prime}\hookrightarrow
                   {\Bbb P}^2\times{\Bbb P}^2\times{\Bbb P}^1,
                   {\cal O}(d_1,d_2,d_3)]$.
                } 
   \begin{eqnarray*}
    \lefteqn{[X\hookrightarrow
               {\Bbb P}^2\times{\Bbb P}^2\times{\Bbb P}^1]}\\[.6ex]
      & = & [{\Bbb P}^2\times{\Bbb P}^2\times{\Bbb P}^1
              \stackrel{id}{\rightarrow}
              {\Bbb P}^2\times{\Bbb P}^2\times{\Bbb P}^1;
             {\cal O}(2,2,0), {\cal O}(1,1,2)]\\[.6ex]
      & = & \tilde{c}_1({\cal O}(1,1,2))
             \circ\tilde{c}_1({\cal O}(2,2,0))
              ([{\Bbb P}^2\times{\Bbb P}^2\times{\Bbb P}^1
                 \stackrel{id}{\rightarrow}
                 {\Bbb P}^2\times{\Bbb P}^2\times{\Bbb P}^1])\\[.6ex]
      & = & F(\tilde{c}_1({\cal O}(1,1,0)),
                \tilde{c}_1({\cal O}(0,0,2)))\\
         && \hspace{5em}
            \circ\, F(\tilde{c}_1({\cal O}(1,1,0)),
                      \tilde{c}_1({\cal O}(1,1,0)))
            ([{\Bbb P}^2\times{\Bbb P}^2\times{\Bbb P}^1
                 \stackrel{id}{\rightarrow}
                 {\Bbb P}^2\times{\Bbb P}^2\times{\Bbb P}^1])\\[.6ex]
      & = & F(\tilde{c}_1({\cal O}(1,1,0)),
              \tilde{c}_1({\cal O}(0,0,2)))\\
         && \hspace{5em}
            \left(
                F(\tilde{c}_1({\cal O}(1,1)),
                  \tilde{c}_1({\cal O}(1,1)))
            ([{\Bbb P}^2\times{\Bbb P}^2
                 \stackrel{id}{\rightarrow}
                 {\Bbb P}^2\times{\Bbb P}^2])\, \times\,
            [{\Bbb P}^1\stackrel{id}{\rightarrow}{\Bbb P}^1]
            \right)
   \end{eqnarray*}
   in
   $\Omega_3({\Bbb P}^2\times{\Bbb P}^2\times{\Bbb P}^1)
    =\omega_3({\Bbb P}^2\times{\Bbb P}^2\times{\Bbb P}^1)\,$,
   where $F=F_{\Omega_{\ast}({\Bbb C})}$
    is the (commutative) formal group law (of rank $1$)
    with coefficients in
    $\Omega_{\ast}({\Bbb C})=\omega_{\ast}({\Bbb C})$.
  This implies that
   $$
    [X\hookrightarrow {\Bbb P}^2\times{\Bbb P}^2\times{\Bbb P}^1]\;
     \in \; \iota_{\ast}(\Omega_{\ast}(H_{2,2}\times{\Bbb P}^1))
     \subset \;
      \Omega_{\ast}({\Bbb P}^2\times{\Bbb P}^2\times{\Bbb P}^1)\,,
   $$
   where
    $H_{2,2}\hookrightarrow {\Bbb P}^2\times{\Bbb P}^2$
     is a Milnor hypersurface in ${\Bbb P}^2\times{\Bbb P}^2$
     described by the zero-locus of a section of ${\cal O}(1,1)$
     that is transverse to the zero-section
     (cf.\ [L-M: Sec.~2.5.3])
       and
    $\iota: H_{2,2}\times {\Bbb P}^1
      \hookrightarrow {\Bbb P}^2\times{\Bbb P}^2\times{\Bbb P}^1$
     is the built-in inclusion.
  It follows that
   $$
    [X\hookrightarrow {\Bbb P}^2\times{\Bbb P}^2\times{\Bbb P}^1,
     {\cal O}(d_1,d_2,d_3)|_X]\;
     \in \; \iota_{\ast}(\omega_{\ast,1}(H_{2,2}\times{\Bbb P}^1))
     \subset \;
      \omega_{\ast,1}({\Bbb P}^2\times{\Bbb P}^2\times{\Bbb P}^1)
   $$
   as well.
  Since $H_{2,2}\times {\Bbb P}^1$ is invariant under an automorphism
   of ${\Bbb P}^2\times{\Bbb P}^2\times{\Bbb P}^1$
   that interchanges the first two ${\Bbb P}^2$-factors and
    leaves the third ${\Bbb P}^1$-factor fixed,
  $$
   {\cal O}(d_1,d_2,d_3)|_{H_{2,2}\times{\Bbb P}^1}\;
    \simeq\; {\cal O}(d_2,d_1,d_3)|_{H_{2,2}\times{\Bbb P}^1}
  $$
   as line bundles on $H_{2,2}\times{\Bbb P}^1$,
   for all $(d_1,d_2,d_3)\in {\Bbb Z}^{\oplus 3}$.
  Furthermore,
  cobordisms among classes
   $[f_i:X_i\rightarrow {\Bbb P}^2\times{\Bbb P}^2\times{\Bbb P}^1]$'s
    in $\omega_{\ast}({\Bbb P}^2\times{\Bbb P}^2\times{\Bbb P}^1)$
   induce canonically cobordisms among classes
    $[f_i:X_i\rightarrow {\Bbb P}^2\times{\Bbb P}^2\times{\Bbb P}^1,
      L|_{X_i}]$'s
    in $\omega_{\ast,1}({\Bbb P}^2\times{\Bbb P}^2\times{\Bbb P}^1)$
    and, similarly,
  cobordisms among classes
   $[f_i:X_i\rightarrow H_{2,2}\times{\Bbb P}^1]$'s
    in $\omega_{\ast}(H_{2,2}\times{\Bbb P}^1)$
   induce canonically cobordisms among classes
    $[f_i:X_i\rightarrow H_{2,2}\times{\Bbb P}^1,L|_{X_i}]$'s
    in $\omega_{\ast,1}(H_{2,2}\times{\Bbb P}^1)$,
   where $L$ is a fixed line bundle on
    ${\Bbb P}^2\times{\Bbb P}^2\times{\Bbb P}^1$ in both situations.
  The lemma follows.

 \end{proof}


 \noindent $(b)$ {\it Chern invariants of line bundles.}\hspace{1em}
  Let $E={\cal O}(d_1,d_2,d_2)|_X$.
  Then the Chern invariants of $[X,E]\in\omega_{3,1}({\Bbb C})$
   can be computed straightforwardly via the intersection product
   in the Chow ring
   $A^{\ast}({\Bbb P}^2\times{\Bbb P}^2\times{\Bbb P}^1)
    \simeq{\Bbb Z}[H_1,H_2,H_3]/(H_1^3,H_2^3,H_3^2)$
   -- with\footnote{Here, we represent a cycle on $X$ by its image
                     on ${\Bbb P}^2\times{\Bbb P}^2\times{\Bbb P}^1$
                     via the built-in inclusion
                     $X\hookrightarrow
                      {\Bbb P}^2\times{\Bbb P}^2\times{\Bbb P}^1$.
                    Similarly for Part (c).}
    $[X]=(2H_1+2H_2)\cdot(H_1+H_2+2H_3)$,
    $c(X)=[X]
          (1+H_1)^3(1+H_2)^3(1+H_3)^2/((1+2H_1+2H_2)(1+H_1+H_2+2H_3))$
     via the adjunction formula,
    $c(E)=[X](1+d_1H_1+d_2H_2+d_3H_3)$, and
    the intersection product
     $H_1\cdot H_1\cdot H_2\cdot H_2\cdot H_3=1$ --
  which gives:\footnote{Part
                        of the computations in Part (b) and Part (c)
                        are aided by {\sl Mathematica} and {\sl Maple}.}
  $$
   \int_X c_3(T_X)\;=\; -96\,,
  $$
  $$
   \int_X c_2(T_X)c_1(E)\;=\; 36\,(d_1+d_2)\,+\,24\,d_3\,,
  $$
  $$
   \int_X c_1(E)^3\;
   =\; 12\,d_1d_2(d_1+d_2)\,+\, 6\,((d_1+d_2)^2+2d_1d_2)d_3\,,
  $$
  $$
   \int_X c_2(T_X)c_1(T_X)\;
    =\; \int_X c_1(T_X)^3\;
    =\; \int_X c_1(T_X)^2c_1(E)\;
    =\; \int_X c_1(T_X)c_1(E)^2\;
    =\; 0\,.
  $$
 Note that these invariants are all functions of $(d_1+d_2, d_1d_2)$,
  which is consistent with Part (a) Lemma under Theorem~2.1.

 \bigskip

 \noindent $(c)$
 {\it Chern invariants of rank-$2$ vector bundles.}\hspace{1em}
  Let $E={\cal O}(a_1,a_2,a_3)|_X\oplus{\cal O}(b_1,b_2,b_3)|_X$.
  Then the Chern invariants of $[X,E]\in\omega_{3,2}({\Bbb C})$
   can be computed similarly as in Part (b)
   -- with now
   $c(E)=[X](1+a_1H_1+a_2H_2+a_3H_3)(1+b_1H_1+b_2H_2+b_3H_3)$ --
  which gives:
  $$
   \int_X c_3(T_X)\;=\; -96\,,
  $$
  \begin{eqnarray*}
   \int_X c_2(T_X)c_1(E)\;
    =\; 36\,(a_1+b_1+a_2+b_2)\,+\, 24\,(a_3+b_3)\,,
  \end{eqnarray*}
  \begin{eqnarray*}
   \lefteqn{\int_X c_2(E)c_1(E)} \\
    & = &  2\,(a_2b_2(a_3+b_3)+(a_2+b_2)(a_2b_3+a_3b_2))\,
       +\, 2\,((a_1+b_1)(a_1b_3+a_3b_1)+ a_1b_1(a_3+b_3))         \\
    && +\, 4\,((a_1+b_1)(a_2b_3+a_3b_2)+(a_1b_2+a_2b_1)(a_3+b_3)
                                       +(a_1b_3+a_3b_1)(a_2+b_2)) \\
    && +\, 4\,((a_1+b_1)a_2b_2 + (a_1b_2+a_2b_1)(a_2+b_2))\,
       +\, 4\,((a_1+b_1)(a_1b_2+a_2b_1)+a_1b_1(a_2+b_2))\,,
  \end{eqnarray*}
  \begin{eqnarray*}
   \lefteqn{\int_X c_1(E)^3\;
    =\;   12\,(a_1+b_1)^2(a_2+b_2)\,+\,12\,(a_1+b_1)(a_2+b_2)^2\,
          +\, 6\,(a_1+b_1)^2(a_3+b_3)} \\
    && \hspace{4.8em}
       +\, 6\,(a_2+b_2)^2(a_3+b_3)\,+\,24\,(a_1+b_1)(a_2+b_2)(a_3+b_3)\,,
       \hspace{7em}
  \end{eqnarray*}
  \begin{eqnarray*}
   \lefteqn{\int_X c_2(T_X)c_1(T_X)\;
         =\; \int_X c_1(T_X)^3}\\
    && =\; \int_X c_1(T_X)^2c_1(E)\;
       =\; \int_X c_1(T_X)c_2(E)\;
       =\; \int_X c_1(T_X)c_1(E)^2\;
       =\; 0\,.
  \end{eqnarray*}

 \bigskip

 \noindent
 $(d)$ {\it Noncompatibility with direct sum.}\hspace{1em}
 Let
  $$
   L_1\; =\; {\cal O}(d_1,d_2,d_3)|_X\,,\;\;
   L_2\; =\; {\cal O}(d_2,d_1,d_3)|_X\,,\;\; \mbox{and}\;\;
   L_3\; =\; {\cal O}(e_1,e_2,e_3)|_X\,.
  $$
 Then,
  $[X,L_1]=[X,L_2]$ in $\omega_{3,1}({\Bbb C})$ from Part (a) Lemma.
 However, let
  $$
   E_1\; :=\;  L_1\oplus L_3\; \hspace{1em}\mbox{and}\hspace{1em}
   E_2\; :=\;  L_2\oplus L_3\,.
  $$
 Then, it follows from Part (c) that, for example,
  \begin{eqnarray*}
   \lefteqn{\int_X c_2(E_1)c_1(E_1)\,-\,\int_X c_2(E_2)c_1(E_2)}\\
   & = &  -\,4\,(d_1-d_2)(e_1-e_2)\, ((d_1+d_2+d_3)+(e_1+e_2+e_3))
  \end{eqnarray*}
  after simplification.
 This shows that:
  \begin{itemize}
   \item[$\cdot$] {\it
    While $[X,L_1]=[X,L_2]$ (and $[X,L_3]=[X,L_3]$)
     in $\omega_{3,1}({\Bbb C})$,
    $$
     [X,L_1\oplus L_3]\;\ne\; [X,L_2\oplus L_3]
    $$
     in $\omega_{3,2}({\Bbb C})$, for general choices of
     $((d_1,d_2,d_3), (e_1,e_2,e_3))\in {\Bbb Z}^{\oplus 6}$.}
  \end{itemize}

 \bigskip

 \noindent
 $(e)$ {\it Noncompatibility with tensor product.}\hspace{1em}
 $L_1$, $L_2$, and $L_3$ as in Part (d).
 Let
  $$
   \begin{array}{lcl}
    E_1  & :=
         & L_1\otimes L_3\; =\; {\cal O}(d_1+e_1,d_2+e_2,d_3+e_3)|_X\,,
                                                           \\[1.2ex]
    E_2  & :=
         & L_2\otimes L_3\; =\; {\cal O}(d_2+e_1,d_1+e_2,d_3+e_3)|_X\,.
   \end{array}
  $$
 Then, it follows from Part (b) that, for example,
  $$
   \int_X c_1(E_1)^3\,-\,\int_X c_1(E_2)^3\;
   =\; -12\,(d_1-d_2)(e_1-e_2)(d_3+e_3)
  $$
  after simplification.
 This shows that:
  \begin{itemize}
   \item[$\cdot$] {\it
    While $[X,L_1]=[X,L_2]$ (and $[X,L_3]=[X,L_3]$)
     in $\omega_{3,1}({\Bbb C})$,
    $$
     [X,L_1\otimes L_3]\;\ne\; [X,L_2\otimes L_3]
    $$
     in $\omega_{3,1}({\Bbb C})$, for general choices of
     $((d_1,d_2,d_3), (e_1,e_2,e_3))\in {\Bbb Z}^{\oplus 6}$.}
  \end{itemize}

 \bigskip

 \noindent
 $(f)$ {\it Noncompatibility with {\it ${\cal H}$om}.}\hspace{1em}
 $L_1$, $L_2$, and $L_3$ as in Part (d).
 Let
  $$
   \begin{array}{lcl}
    E_1 & :=
        & L_1^{\vee}\otimes L_3\;
          =\; {\cal O}(-d_1+e_1,-d_2+e_2,-d_3+e_3)|_X\,, \\[1.2ex]
    E_2 & :=
        & L_2^{\vee}\otimes L_3\;
          =\; {\cal O}(-d_2+e_1,-d_1+e_2,d_3+e_3)|_X\,.
   \end{array}
  $$
 Then, it follows from Part (e) that
  $$
   \int_X c_1(E_1)^3\,-\,\int_X c_1(E_2)^3\;
   =\; 12\,(d_1-d_2)(e_1-e_2)(d_3+e_3)\,.
  $$
 Similarly, let
  $$
   \begin{array}{lcl}
    E_1^{\prime} & :=
     & L_3^{\vee}\otimes L_1\;
       =\; {\cal O}(d_1-e_1,d_2-e_2,d_3-e_3)|_X\,,  \\[1.2ex]                           
    E_2^{\prime} & :=
     & L_3^{\vee}\otimes L_2\;
       =\; {\cal O}(d_2-e_1,d_1-e_2,d_3-e_3)|_X\,.
   \end{array}
  $$
 Then, it follows from Part (e) that
  $$
   \int_X c_1(E_1^{\prime})^3\,-\,\int_X c_1(E_2^{\prime})^3\;
   =\; 12\,(d_1-d_2)(e_1-e_2)(d_3-e_3)\,.
  $$
 This shows that:
  \begin{itemize}
   \item[$\cdot$] {\it
    While $[X,L_1]=[X,L_2]$ (and $[X,L_3]=[X,L_3]$)
     in $\omega_{3,1}({\Bbb C})$,
    $$
     [X,\Homsheaf_X(L_1,L_3)]\;\ne\; [X,\Homsheaf_X(L_2,L_3)]
    $$
     and
    $$
     [X,\Homsheaf_X(L_3,L_1)]\;\ne\; [X,\Homsheaf_X(L_3,L_2)]
    $$
     in $\omega_{3,1}({\Bbb C})$, for general choices of
     $((d_1,d_2,d_3), (e_1,e_2,e_3))\in {\Bbb Z}^{\oplus 6}$.}
  \end{itemize}

 \noindent
 This concludes the example.
}\end{sexample}

\bigskip

\begin{flushleft}
{\bf Question:
 Refined and higher algebraic cobordisms of vector bundles on varieties?}
\end{flushleft}
In the construction of Example~2.3,
a reason why basic operations on vector bundles in general
 are not compatible with algebraic cobordisms
 is that
 \begin{itemize}
  \item[$\cdot$] {\it
   Isomorphisms of bundles on varieties
    in general do not fit into a cobordism family of isomorphisms
    of bundles on varieties that contains an identity map.}
  \end{itemize}
Comparing this phenomenon with the construction of $L$-group
 and its higher extension $L^{\bullet}$
  via the simplicial setting of the related cobordism theory
  for $L$-group (cf.\ [Lu]),
it is natural to pose the following question:

\begin{squestion}
 {\bf [refined and higher algebraic cobordism].}
 Is there a refined construction of algebraic cobordism of bundles
  on varieties that make it compatible with basic operations on bundles?
 Is there a simplicial construction that gives rise to
  higher algebraic cobordisms $\Omega_{\ast}^{\,\bullet}$
   (resp.\ $\omega_{\ast}^{\,\bullet}$,
           $\omega_{\ast,\ast}^{\,\bullet}$)
   of varieties or
   of vector bundles on varieties
  that homotopy-theoretically extends $\Omega_{\ast}$
    of Levine and Morel {\rm [L-M]}
    (resp.\
     $\omega_{\ast}$ of Levine and Pandharipande {\rm [Lev-P]},
     $\omega_{\ast,\ast}$ of Lee and Pandharipande  {\rm [Lee-P]})?
\end{squestion}

\bigskip
\bigskip

Finally,
we remark that the note
 is initiated from a discussion in [Tz2] and
 motivated by
 issues in enumerative geometry of BPS bound states of D-branes
  discussed in numerous string-theory literatures
  -- for example,
     [D-M] (2007) of Frederik Denef and Gregory Moore --
  and
 the notion of configurations in abelian categories
  in [Jo] (2003--2005) of Dominic Joyce.
This prepares us in part for related themes in a D-brane project
 (cf.\ [L-Y] (D(6)) for a partial review).

\newpage
\baselineskip 13pt

\end{document}